\documentclass{amsart}
\usepackage[utf8]{inputenc}
\usepackage[all]{xy}
\usepackage{amsmath,amssymb}
\usepackage{enumerate}
\newtheorem{theorem}{Theorem}[section]

\newtheorem{proposition}[theorem]{Proposition}
\newtheorem{definition}{Definition}
\newtheorem{corollary}[theorem]{Corollary}
\newtheorem{lemma}{Lemma}
\newtheorem*{remark}{Remark}

\DeclareMathOperator{\Ricci}{Ric}

\def\rm#1{\mathrm{#1}}
\def\cal#1{\mathcal{#1}}
\def\bb#1{\mathbb{#1}}
\def\lie#1{\mathfrak{#1}}
\def\lr#1{\left\langle #1\right\rangle}
\usepackage[all]{xy}

\newcommand{\halmos}{\hfill $\;\;\;\Box$\\}

\newcommand{\h}{\frac{1}{2}}

\DeclareMathOperator{\id}{id}

\DeclareMathOperator{\Ad}{Ad}

\newtheorem{claim}{Claim}

\title{On the geometry of some Equivariantly related manifolds}
\author{ Llohann D. Sperança \and Leonardo F. Cavenaghi}
\date{}

\begin{document}
	
	\begin{abstract}
		We provide a topological procedure to obtain geometric realizations of both classical and `exotic' $G$-manifolds, such as spheres, bundles over spheres and Kervaire manifolds. As an application, we apply the process known as Cheeger deformations to produce new metrics of both positive Ricci and almost non-negative curvature on such objects.
	\end{abstract}
	\thanks{The first author is financially supported by CNPq, grant number 131875/2016-7. The second author was financially supported by FAPESP, grant numbers  2009/07953-8 and 2012/25409-6 and CNPq 404266/2016-9.}
	
	\maketitle

	\section{Introduction}
	
	Compared with other curvature conditions, examples of manifolds with positive sectional curvature are still sparse in literature (see \cite{Zillerpositive} for a survey). For instance, there are still no obstructions that distinguish the class of simply connected compact positively curved manifolds from   non-negatively curved ones. In parallel to this situation, it is still not known the existence of metrics of positive sectional curvature in interesting manifolds, such as exotic spheres and sphere bundles over spheres. 
	
	Easier to tackle is the existence of metrics of positive Ricci and positive scalar curvatures. Many geometric/topological constructions are well succeed in this aim. For instance, the topological constructions  in \cite{crowley2017positive,gromov1980classification,wraith1997,wraith2007new} and the constructions based on symmetries  in \cite{lawson-yau,gz,searle2015lift}.

	%	Here we present geometric realizations of these examples together with a reasonable set of  metrics with positive Ricci curvature, related to standard metrics on standard manifolds (such as round spheres).

	Here we provide a common ground for \cite{DRS,gromoll1974exotic,sperancca2016pulling}, geometric realizing `exotic' manifolds as isometric quotients of principal bundles over `standard' ones. The realization is incarnated in the form of a cross-diagram:
	\begin{equation}\label{eq:CD}
	\begin{xy}\xymatrix{& G\ar@{..}[d]^{\bullet} & \\ G\ar@{..}[r]^{\star} & P\ar[d]^{\pi}\ar[r]^{\pi'} &M'\\ &M&}\end{xy}
	\end{equation}
	Special properties of the constructed diagrams allow one to compare the geometries of $M$ and $M'$ (see section \ref{ssec:notation} and \ref{sec:geometry}). Former instances of diagram \eqref{eq:CD} are found in \cite{duran2001pointed,DPR,DRS,speranca2016pulling}, mainly exploiting the differential topology of exotic spheres. 
	%	Therefore, it might offer an alternative approach to seek  existence or non-existence of rigid properties in specific manifolds, such as positive sectional curvature on exotic spheres. For instance, interesting manifolds can be realized as $M'$, keeping $M$ something as simple as a round sphere.
	On the other hand, the construction explicitly realizes \emph{Morita equivalences} between related transformation groupoids\footnote{The authors thank O. Brahic and C. Ortiz for pointing it out} which might be of independent interest.
	
	As an application, we establish the existence of metrics of positive Ricci curvature and almost non-negative curvature in new examples (to the best of our knowledge). The construction works specially realizing some connected sums (see section \ref{ssec:connectedsum} or Theorem \ref{thm:main}).
	
	We recall that a manifold $M$ has almost non-negative sectional curvature if it admits a sequence of Riemannian metrics $\{g_n\}_{n\in\bb N}$ satisfying $\sec_{g_n}\geq -1/n$ and $\rm{diam}_{g_n}\geq 1/n$ (see \cite{kapovitch2010nilpotency} for recent developments). 
	
	\begin{theorem}\label{thm:main}
		Let $\Sigma^7$, $\Sigma^8$ be any homotopy sphere in dimensions 7 and 8, $\Sigma^{10}$ any homotopy 10-sphere which bounds a spin manifold and $\Sigma^{4m{+}1}$ the Kervaire sphere in dimension $4m{+}1$. Then the manifolds below admits a sequence $\{g_n\}_{n\in\bb N}$ of metrics with positive Ricci curvature satisfying $\sec_{g_n}\geq -1/n$ and $\rm{diam}_{g_n}\geq 1/n$:
		\begin{enumerate}[$(i)$]
			\item $M^7\#\Sigma^7$, where $M^7$ is any 3-sphere bundle over $S^4$, 
			\item $M^8\#\Sigma^8$, where $M^8$ is either a 3-sphere bundle over $S^5$ or a 4-sphere bundle over $S^4$
			\item $M^{10}\#\Sigma^{10}$, where 
			\begin{enumerate}[$(a)$]
				\item $(M^8\times S^2)$,  $M^8$  as in item $(2)$
				\item  $M^{10}$ is any 3-sphere bundle over $S^7$, 5-sphere bundle over $S^5$ or 6-sphere bundle over $S^4$				
			\end{enumerate}
			\item $M^{4m+1}\#\Sigma^{4m+1}$ where \label{item:5} 
			\begin{enumerate}[$(a)$]
				\item $S^{2m}\cdots M^{4m+1}\to S^{2m+1}$ is a sphere bundle representing any multiple of the unitary tangent of $S^{2m+1}$
				\item $\bb CP^{m}\cdots M^{4m+1}\to S^{2m+1}$ is the $\bb CP^m$-bundle associated to any  multiple of $U(m)\cdots U(m+1)\to S^{2m+1}$
				\item $\frac{U(m+2)}{SU(2)\times U(m)}$	
			\end{enumerate}
			\item $(M^{8m+k}\times N^{5-k})\#\Sigma^{8m+5}$ where $N^{5-k}$ is any manifold with positive Ricci curvature and
			\begin{enumerate}[$(a)$]
				\item $S^{4m+k}\cdots M^{4m+k}\to S^{4m+1}$ is a sphere bundle representing the $k$-th (fiberwise) suspension of any multiple of the unitary tangent of $S^{4m+1}$
				\item $k=1$, $\bb HP^{m}\cdots M^{8m+1}\to S^{4m+1}$ is the $\bb HP^m$-bundle associated to any  multiple of $Sp(m)\cdots Sp(m+1)\to S^{4m+1}$
				\item $k=0$, $M=\frac{Sp(m+2)}{Sp(2)\times Sp(m)}$
				\item $k=1$, $M=M^{8m+1}$ as in item  $(\ref{item:5})$
			\end{enumerate}
		\end{enumerate}	
	\end{theorem}

	We recall that a homotopy $n$-sphere is a manifold homotopy equivalent to the unitary sphere of $\bb R^{n+1}$. For $n>4$, homotopy spheres are homeomorphic to standard spheres (Smale \cite{smale}). It is \textit{exotic} if it is not diffeomorphic to the standard sphere. As computed by Kervaire and Milnor \cite{km}, there are (up to orientation) exactly 14 exotic spheres in dimension 7, one in dimension 8 and one in dimension 10 which bounds a Spin manifold (there are two more that does not). Kervaire spheres are the boundary of $P^{2n}(\mathbf{A}_2)$ (in Bredon \cite{brebook} notation), the plumbing of two copies on the (disc bundle associated to the) tangent bundle of $S^n$, $n=2m+1$ (see section \ref{section:examples} for a more detailed description). The Kervaire sphere is known to be exotic for all but finitely many (known) values of $m$ (see Hill--Hopkins--Ravenel \cite{hill2016nonexistence} and Wang--Xu \cite{wang2016triviality}). When $n$ is even, the resulting boundary, $\Sigma^{4m-1}=\partial P^{4m}(\mathbf{A}_2)$, has $H_{2m}(\Sigma^{4m-1})\cong \bb Z_3$.
	
	As an extra contribution, we provide metrics of positive Ricci curvature on quotients of $\Sigma^{4m-1}$. Whenever $n$ is even or odd, $\partial P^{2n}(\mathbf{A}_2)$ posses a (bi-axial) $O(n)$-action which restricts to a free $SO(2)$ (respectively, $SU(2)$) action when $n$ is even (respectively, when $n=0\mod 4$). We denote the quotient by $\Sigma P\bb C^n$ (respectively, $\Sigma P\bb H^{n/2}$). 
	$\Sigma P\bb C^n$ (respectively, $\Sigma P\bb H^{n/2}$) is an $U(n)$-manifold (respectively, a $Sp(n/2)$-manifold).
	
	\begin{theorem}
		$\Sigma P\bb C^n$ and $\Sigma P\bb H^{n/2}$ posses invariant metrics with positive Ricci curvature.
	\end{theorem}

	The restriction/variety of examples in Theorem \ref{thm:main} is related to the presence of symmetries and explicit realizations of spheres. To give a more precise statement, we fix the representations below:
	%	related to isotropy representations of fixed points in $\Sigma^7,~\Sigma^8,~\Sigma^{10},\Sigma^{4m+1}$:
	\begin{itemize}
		\item[$(\rho_7)$] $n=7$, $G=S^3$: $\rho=\rho_{1}\oplus \rho_1\oplus \rho_0$, where $\rho_1:S^3\to O(3)$ is the  representation induced by the double cover $S^3\to SO(3)$ and $\rho_0:S^3\to SO(1)$ is the trivial one 
		\item[$(\rho_8)$] $n=8$, $G=S^3$: $\rho=\rho_{\h}\oplus \rho_1\oplus \rho_0$, where $\rho_{\h}:S^3\to O(4)$ is the representation induced by right (or left) multiplication by a quaternion
		\item[$(\rho_{10})$] $n=10$, $G=S^3$: $\rho=\rho_{\h}\oplus \rho_1\oplus 3\rho_0$ or $\rho=\rho_1\oplus \rho_1\oplus \rho_1\oplus \rho_0$
		\item[$(\rho_{4m+1})$] $n=4m+1$, $G=U(m)$: $\rho=\rho_{U(m)}\oplus\rho_{U(m)}\oplus\rho_0$ where $\rho_{U(m)}:U(m)\to O(2m)$ is the standard representation
		\item[$(\rho_{8m+5})$] $n=8m+5$, $G=Sp(m)$: $\rho=\rho_{Sp(m)}\oplus\rho_{Sp(m)}\oplus 5\rho_0$ where $\rho_{Sp(m)}:Sp(m)\to O(4m)$ is the standard representation
	\end{itemize}

	Given a $G$-invariant metric on $M$, we recall that $M/G$ is a metric space with orbit distance. The $G$-orbits have constant dimension on an open and dense set of $M^*$, called the principal part (see Bredon \cite{brebook}). The subset $M^*/G$ has a natural manifold structure that makes the quotient $M^*\to M^*/G$ a Riemannian submersion.
	
	\begin{theorem}\label{thm:main2}
		Let $M^n$ be a $G$-manifold that posses a fixed point whose isotropy representation is $(\rho_n)$. Then: 
		\begin{enumerate}
			\item if the orbits on $M$ have finite fundamental group and $M$ posses a $G$-invariant metric such that  $Ricci_{M^*/G}\geq1$, then $M^n\#\Sigma^n$ has a metric with positive Ricci curvature 
			\item if  $M$ posses a family of $G$-invariant metrics such that $M^*/G$ has almost non-negative curvature, then $M^n\#\Sigma^n$ has almost non-negative curvature
			\item if $M$ posses a family of metrics satisfying (1) and (2), then $M^n\#\Sigma^n$ admits a family of metrics $\{g_n\}_{n\in\bb N}$ with positive Ricci curvature satisfying $\sec_{g_n}\geq -1/n$ and $\rm{diam}_{g_n}\geq 1/n$ 
		\end{enumerate}
	\end{theorem}
	%	verify the conditions in Theorems A and B of Searle--Wilhelm \cite{searle2015lift}, therefore 

	For Theorem \ref{thm:main2}, we construct a diagram as in \eqref{eq:CD} with $G$ compact and $M'$ diffeomorphic to $M^n\#\Sigma^n$. For $G$ compact, we prove that, given a $G$-invariant metric on $M$, there is a $G$-invariant metric on $M'$  such that $M'/G$ is isometric to $M/G$ (see section \ref{ssec:notation} for the $G$-action on $M'$). Thus, if $M$ satisfies the hypothesis in Theorems A, B or C in Searle--Wilhelm \cite{searle2015lift}, so does $M'$. Using diagram \eqref{eq:CD}, we also provide a self-contained proof of Theorem \ref{thm:main2} (the greater machinery in \cite{searle2015lift} can be avoided in Theorem \ref{thm:main} since diagram \eqref{eq:CD} gives a better control of the geometry on the non-principal part of $M$).
	
	Some prior works related to Theorem \ref{thm:main}: the works of Nash \cite{nash1979positive} and Poor \cite{poor1975some} provides metrics of positive Ricci curvature on sphere bundles over spheres. The procedure in Fukaya--Yamaguchi \cite{fukaya-yamaguchi} is consistent with Nash \cite{nash1979positive} and Poor \cite{poor1975some}, thus, providing metrics which are simultaneously Ricci positive and almost non-negative on sphere bundles over spheres. Crowley--Wraith \cite{crowley2017positive} considers positive Ricci curvature on connected sums of highly connected odd-dimension manifolds with exotic spheres, covering the Ricci curvature in example $(i)$ and (possibly partially) in example $(iv)$. Almost non-negative curvature in homotopy 7-spheres was established by Searle--Wilhelm \cite{searle2015lift}.

	The connected sum $M^n\#\Sigma^n$ does not always results in a new manifold. Following de Sapio \cite{DeSII,de1969manifolds}  (recall from \cite[page 3190]{speranca2016pulling} and \cite[Section V.8]{brebook} that, in de Sapio's notation, $\Sigma^8\in\sigma(\pi_3SO(4),\pi_4SO(3))$, $\Sigma^{10}$ is an order 3 generator of $\sigma(\pi3SO(6),\pi_6SO(3))$ and $\Sigma^{4m+1}=\sigma(\tau_{2m},\tau_{2m})$, where $\tau_n:S^n\to SO(n+1)$ is a transition function for the tangent bundle of $S^{n+1}$), one concludes that $M^n\#\Sigma^n\cong M^n$ if $M^n$ is the non-trivial $SU(2)$-principal bundle over $S^5$, the $SU(2)$-principal bundle over $S^7$ corresponding to a generator of $\pi_6SU(2)\cong \bb Z_{12}$ and the unit sphere bundle $T_1S^{2m+1}$. $M^n\#\Sigma^n\ncong M^n$ for the following $M^n$: any product of standard spheres; any 4-sphere bundle over $S^4$; $\Sigma^8\times S^2$; the $SU(2)$-principal bundles corresponding to three times a generator in $\pi_6SU(2)$. All other cases are undecidable within the authors reach.

	\subsection{Cross Diagrams and Notation} \label{ssec:notation}
	The present construction encompasses the constructions in \cite{gromoll1974exotic,DPR,DRS,sperancca2016pulling}. The aim is to produce a diagram as in \eqref{eq:CD} out of a given  $G$-manifold $M$ and equivariant bundle data (see Definition \ref{defn:howtoconstruct}  for details). For short, we denote diagram \eqref{eq:CD} by $M\stackrel{\pi}{\leftarrow}P\stackrel{\pi'}{\to}M'$.
	
	In diagram \eqref{eq:CD}, $P$ is a \textit{(special) $G$-$G$-bundle}: following \cite{lashof1982equivariant}, a $G$-$G$-bundle $\pi:P\to M$ is a principal $G$-bundle equipped with an extra $G$-action, which we call the \textit{$\star$-action}. Both $G$-actions on $P$ are assumed to commute, making $P$ a $G{\times }G$-manifold. When dealing with the $G\times G$-action, we use $G{\times}\{\id\}$ as the $\star$-action and $\{\id\}{\times} G$ as the principal action of $\pi$. We denote the $\star$-action (respectively, the principal action of $\pi$) by left juxtaposition (respectively, right juxtaposition). That is, for every $(r,s,p)\in G\times G\times P$, $(r,s,p)\mapsto rps^{-1}$ (we use the inverse of $s$ to keep the `left action' convention).
	
	Since the $\star$-action commutes with the principal action of $\pi$, it descends to a $G$-action on $M$ which we denote by $(r,x)\mapsto r\cdot x$. The action makes $\pi$ equivariant: $\pi(rps^{-1})=r\pi(p)$. In particular, if $(r,s)\in G\times G$ fix $p$, $r$ fixes $\pi(p)$. On the other hand, if $r\in G$ fixes $x\in M$, since the principal action of $\pi$ acts freely and transitively on the fibers of $\pi$, for every $p\in \pi^{-1}(x)$, there is a unique $s_p\in G$ such that $rps_p^{-1}=p$. Note that $s_{pg^{-1}}=gs_pg^{-1}$. We denote by $(G\times G)_p$ and $G_x$ the isotropy subgroups on $P$ and $M$.

	As in \cite{speranca2016pulling}, we consider only {special} $G$-$G$-bundles (and, by doing so, omit the adjective `special') by imposing two further conditions: the $\star$-action must be free; for every $p\in P$, there is $g\in G$ such that\footnote{In practice we use Definition \ref{defn:howtoconstruct}, based on transition functions. An equivalence between \eqref{eq:isotropy}  and Definition \ref{defn:howtoconstruct} follows from Proposition \ref{prop:fibradonormal}.}
	\begin{equation}\label{eq:isotropy}
	(G\times G)_p=\{(h,ghg^{-1})~|~h\in G_{\pi(p)} \}.
	\end{equation}
	Observe that, if $M\stackrel{\pi}{\leftarrow}P\stackrel{\pi'}{\to}M'$ is a special $G$-$G$-bundle, so it is $M'\stackrel{\pi'}{\leftarrow}P\stackrel{\pi}{\to}M$ (by interchanging the wholes of the two actions). In particular, the principal action of $\pi$ defines a $G$-action on $M'$.

	Most manifolds will be constructed by patching together subsets. For $A_i\subset X_i$, $i\in\Lambda$, and  bijections $f_{ij}:A_i\to A_j$ satisfying $f_{ik}f_{kj}f_{ji}(x)=f_{ii}(x)=x$ (for any $x$ it make sense), we denote
	$\{f_{ij}:A_i\to A_j\}_{i,j\in\Lambda}$, we denote 
	\begin{equation}
	\cup_{f_{ij}}X_i={{\cup}X_i}\big/{\sim},
	\end{equation} 
	where $\sim$ is the equivalence relation whose non-trivial relations are $A_i\ni x\sim f_{ij}(x)\in A_j$. Of particular interest are the cases of `twisted manifolds' and of principal bundles: let $\{X_i=U_i\}_{i\in\Lambda}$ be an open cover for a manifold $M$ and $f_{ij}:U_i\cap U_j\to U_i\cap U_j$ diffeomorphisms. The new object $\cup_{f_{ij}}U_i$ is a manifold locally resembling $M$, but with possibly different global nature; to define a principal bundle, it is sufficient to provide a collection of \textit{transition functions}. That is,  a collection $\{\phi_{ij}:U_i\cap U_j\to G\}_{i,j\in\Lambda}$ satisfying the \textit{cocycle condition}:
	\begin{equation}\label{eq:cocycle}
	\phi_{ij}\phi_{jk}\phi_{ki}(x) =\phi_{ii}(x)=\rm{id}\in G, ~\forall x \in U_i\cap U_j \cap U_k.
	\end{equation}
	The total space $P=\cup_{f_{ij}}X_i$ is recovered by setting $X_i=U_i\times G$ and $f_{ij}(x,g)=f_{\phi_{ij}}(x,g)=(x,g\phi_{ij}(x))$.  The projection $\pi\,{:}\,P\to M$ and the principal action $(s,p)\mapsto ps^{-1}$ are defined on $U_i\times G$ by
	\begin{equation}
	\pi(x,g)=x,\qquad (x,g)s^{-1}=(x,sg).
	\end{equation}
	Both $\pi$ and the multiplication (as defined above) commute with $f_{\phi_{ij}}$, defining global maps.
	
	Although, only open covers are considered above, the case  where $M$ is the union of two manifolds by their boundaries is present in the examples. Let $M=X\cup Y$ and $f:\partial X\to\partial Y$ a diffeomorphism. The manifold $M'=X\cup_f Y$ admits a unique smooth structure such that the inclusions $X,Y\subset M'$ are smooth. Such smooth structure is obtained through a collaring argument:   using the collaring theorem (Theorem (3.3) in Kosinski \cite{ko}) one can extend $X,Y$ to open manifolds $X',Y'$ (without boundaries) whose 'ends' are diffeomorphic to $\partial X\times (-1,1)$, $\partial Y\times (-1,1)$ with embeddings $j_X:\partial X\times (-1,1)\to X'$, $j_Y:\partial Y\times (-1,1)\to Y'$. Then $M'$ is diffeomorphic to $X'\cup_{f'} Y'$ where $f'(j_X(x,t))=j_Y(f(\tilde x),-t)$. Isotopies of $f$ defines isotopies of $f'$ implying that the manifold structure of $M'$ does not depend on the choices of $j_X,j_Y$ or the isotopic class of $f$.
	
	%	In fact, note that
	%	\[\widehat{\phi}_{ij}(g\cdot x) = \phi_{ij}(g\cdot x)\cdot (g\cdot x) = g\phi_{ij}(x)\cdot((g^{-1}g)\cdot x) = g\phi_{ij}(x)\cdot x = g\widehat{\phi}_{ij}(x).\] 
	
	A sphere bundle $S^{k{-}1}\cdots M\to B$ is called \textit{linear} if $O(k)$ (acting in the usual way on $S^{k{-}1}$) is a structure group. Equivalently, if there is a set of transition functions $\{\phi_{ij}:U_i\cap U_j\to O(n)\}$.
	
	Recall that linear sphere bundles can always be \textit{suspended}: consider $s_k:O(k)\to O(k{+}1)$  the inclusion of $O(k)$ as the subgroup of $O(k{+}1)$ with 1 in the upper-left corner. If $\{\phi_{ij}:U_i\cap U_j\to O(k)\}$ are transition functions for $S^{k-1}\cdots M\stackrel{\pi}{\to}B$, then $\{s_k\phi_{ij}:U_i\cap U_j\to O(k{+}1)\}$ are transition functions for a linear $S^k$-bundle over $B$.
	
	We denote the quaternionic field as $\bb H$ and its subspace of pure imaginary quaternions as $\rm{Im}\bb H$. The norm, inverse and conjugate of a quaternion $x$ are denoted as $|x|,~x^{-1},~\bar x$, respectively.
	
	Some standard spaces (such as discs and spheres) are described in coordinates. When we denote $M\subset V_1\times V_2\times V_3$ we mean that a point in $M$ is a triple $(x_1,x_2,x_3)$ where $x_i\in V_i$ -- we observe there might be relations among the $x_i$'s. For instance, $S^7\subset \bb H\times \bb H$ denotes $S^7=\{(x,y)\in \bb H\times\bb H~|~ |x|^2+|y|^2=1\}$.
	
	We denote by $R_g$ the Riemannian tensor of the metric $g$, adopting the sign convention 
	\[R_g(X,Y)Z=\nabla_X\nabla_YZ-\nabla_Y\nabla_XZ-\nabla_{[X,Y]}Z.\]
	We denote $K_g(X,Y)=g(R_g(X,Y)Y,X)$, the unreduced sectional curvature of $g$ and by $||X\wedge Y||^2_g=||X||^2_g||Y||^2-g(X,Y)^2$. The Ricci tensor of $g$ is denoted by
	\[\Ricci_g(X,Y)=\sum_{i=1}^ng(R(e_i,X)Y,e_i), \]
	where $\{e_1,...,e_n\}$ is an orthonormal basis for $g$. The associated quadratic form is denoted $\Ricci_g(X)=\Ricci_g(X,X)$.
	
	%	The  symbols $\eta_1,\eta_2,\eta_3,\eta_4,w_1,w_2$  will denote elements of the quaternionic field $\bb H$ and $\xi$ an element of the subspace of pure imaginary quaternions $\rm{Im}~\bb H$. $\lambda$ is a real number. Juxtaposition of quaternions will always mean quaternionic multiplication. 
	
	%	We often consider products of manifolds; more specifically, products of discs with spheres, spheres with spheres and manifolds with groups. We denote arbitrary  elements on products of discs with spheres and spheres with spheres as $(x_1,x_2)$. For products of manifolds and groups (which are usually related to trivializations of bundles), we always denote the group element as $g$ (even if it is a quaternion).
	
	%	The letter $p$ is always a point of $P$ and $x'$ of $M'$. We will use $x$ for elements in $M$ or in subsets of $M$; we use $x\in D^7\subset S^7$, for example.  The only exception is $Sp(2)$, where we denote an arbitrary element  as $Q$.

	%	\textbf{Include $\sim$, $\cup_f$ convention, suspensions, bundle conventions }
	
	\subsection*{Acknowledgments} Part of this work was accomplished during the first author's PhD under Prof. A. Rigas and Prof. C. Dur\'an and a postdoc period under Prof. L. A. B. San Martin. The first author thanks all of them.
	
	Both thank Prof. Lino Grama  for suggestions and comments. The second author also thanks Prof. L. Grama for teaching significant part of his geometrical background.
	
	\section{Constructing $G$-$G$-bundles}
	\label{sec:star}

	Consider a $G$-manifold $M$, an open cover $\{U_i\}$ and a set of transition functions $\{\phi_{ij}:U_i\cap U_j\to G\}$. By imposing conditions on $U_i$ and $\phi_{ij}$, one can end up with extra structure on the bundle $\pi:P\to M$ defined by $\{\phi_{ij}\}$.
	
	\begin{definition}\label{defn:howtoconstruct}
		Let $M$ be a $G$-manifold and $\{U_i\}$ a $G$-invariant open cover (i.e., $U_i$ is $G$-invariant for every $i\in\Lambda$). A collection $\{\phi_{ij}:U_i\cap U_j\to G\}$ is called a \textit{$\star$-collection} (or, for short, a $\star$-cocycle) if $\phi_{ij}$ satisfies the cocycle condition \eqref{eq:cocycle} and 
		\begin{equation}\label{eq:conjugation}
		\phi_{ij}(g\cdot x) = g\phi_{ij}(x)g^{-1},
		\end{equation}
		for all $i,j$ and $x\in U_i\cap U_j$.
	\end{definition}
	
	Condition \eqref{eq:conjugation} is just equivariance with respect to conjugation on $G$. Given $\phi_{ij}$, we define the \textit{adjoint map} 
	\begin{align}\widehat{\phi}_{ij} : U_i \cap U_j& \to U_i \cap U_j\nonumber\\
	x& \mapsto \phi_{ij}(x)\cdot x.
	\end{align}
	$\widehat \phi_{ij}$ is a $G$-equivariant diffeomorphism of $U_i\cap U_j$ whenever $\phi_{ij}$ satisfies \eqref{eq:conjugation} (Lemma \ref{lem:reentrance}). 
	
	The core of the paper resides on the next Theorem.

	\begin{theorem}\label{thm:star}
		Let $\pi:P\to M$ be the principal $G$-bundle associated to a $\star$-collection $\{\phi_{ij}:U_i\cap U_j\to G\}$. Then $P$ admits a new action $\star:G\times P\to P$, such that
		\begin{enumerate}[$(i)$]
			\item  $\star$ is free and commutes with the principal action of $\pi$
			\item the quotient $P/\star$ is a $G$-manifold  $G$-equivariantly diffeomorphic to
			\begin{equation*}\label{eq:M'}
			M':= \cup_{{\widehat{\phi}}_{ij}}U_i,
			\end{equation*}
			where the $G$-action on $U_i$ is the restriction of the $G$-action on $M$.
		\end{enumerate}
	\end{theorem}
	%	Note that the action defined on $U_i\subset M$ globally defines an action on $M'$ since $\widehat{\phi}_{ij}$ is $G$-equivariant for every $i,j$. 
	\begin{proof}%	\begin{proof}[Proof of Theorem \ref{thm:star}]
		Let $P=\cup_{f_{\phi_{ij}}} U_i\times G$ be as in section \ref{ssec:notation}. We define the $\star$-action as
		\begin{equation}\label{eq:staraction}
		r(x,g)=(r\cdot x,gr^{-1}).
		\end{equation}
		The action is globally well-defined since, for  $x\in U_i\cap U_j$,
		\[(q\cdot x,gq^{-1}\phi_{ij}(qx)) = (q\cdot x,g\phi_{ij}(x)q^{-1}).\]
		The definition of $M'$ implicitly requires the cocycle condition  $\widehat{\phi}_{jk}\widehat{\phi}_{ij}=\widehat{\phi}_{ik}$. Moreover, its smooth $G$-manifold structure requires  $\widehat{\phi}_{ij}$ to be equivariant diffeomorphisms. Next Lemma guarantees both conditions (see also Bredon \cite[page 49]{brebook}).
		\begin{lemma}\label{lem:reentrance}
			Let $U$ be a smooth $G$-manifold and $\theta, \theta': U \to G$ be smooth maps satisfying \eqref{eq:conjugation}. Then $\widehat{\theta},\widehat{\theta'}:U\to U$ are $G$-equivariant diffeomorphisms and $\widehat{\theta\theta'}=\widehat{\theta'}\widehat{\theta}$, where $\theta\theta': U \to G$ is the pointwise multiplication $\theta \theta'(x) = \theta(x)\theta'(x)$. 
			%			That is, the association $\theta\mapsto\widehat{\theta}$ defines an antihomomorphism from the group of equivariant smooth maps to $G$ to the group of equivariant diffeomorphisms of $U$.
			%			\begin{equation}\label{comportamento}  \widehat{\theta\theta'} = \widehat{\theta'}\widehat{\theta}.\end{equation}
		\end{lemma}
		\begin{proof}
			For $x \in U$,
			\begin{equation*}\label{proof:reentrance}(\widehat{\theta'}\widehat{\theta})(x) = \theta'(\theta(x)x)\theta(x)x= \theta(x)\theta'(x)\theta(x)^{-1}\theta(x)x=\theta(x)\theta'(x)x=\widehat{\theta\theta'}(x).\end{equation*}
			In particular, taking $\theta'(x)=\theta(x)^{-1}$, one gets $\widehat{\theta'}=\widehat{\theta}^{-1}$. $G$-equivariance of $\widehat{\theta}$ is straightforward.
		\end{proof}
		To identify $P/\star$ with $M'$, we define  $\pi':P\to M'$ in each  $U_i\times G$ by
		\begin{equation}\label{eq:pi'}
		\pi'(x,g)=g\cdot x.
		\end{equation}
		(compare Cheeger \cite{cheeger}) $\pi'$ is well defined on $P$ since, for all $x\in U_i\cap U_j$, 
		\[\pi'(x,g\phi_{ij}(x))=g\phi_{ij}(x)x=\phi_{ij}(gx)gx=\widehat{\phi}_{ij}(\pi'(x,g)),\]
		Moreover, $\pi'(rp)=\pi'(p)$ for every $p\in P$ and $\iota:U_i\to P$ defined by 
		\begin{equation}\label{eq:iota}
		\iota_i(x)=(x,\id)
		\end{equation} defines a local section for both $\pi$ and $\pi'$. In particular, $\pi'$ is a submersion whose fibers are the $\star$-orbits.	
	\end{proof}

	\subsection{Induced Bundles}
	\label{section:pullback}
	
	Before proceeding to examples, we provide a method to produce new $G$-$G$-bundles out of old ones. As one can see, a $G$-$G$-bundle naturally lies on the category of $G$-spaces and $G$-equivariant maps. Therefore, one could expect that the pullback construction could be carried out by equivariant maps on  the set of $G$-$G$-bundles. Here we provide details of this procedure (compare \cite[section 2]{speranca2016pulling}).
	
	Let $M, N$ be smooth $G$-manifolds and  $f : N \to M$ a smooth  $G$-equivariant map, i.e., $f(g\cdot x) = g\cdot f(x)$ for every $(x,g) \in N{\times} G$. Let $\{\phi_{ij}:U_i\cap U_j\to G\}$ be a $\star$-collection of transition functions  associated to the covering $M=\cup U_i$ and consider $\pi : P \to M$ the $\star$-bundle associated to such collection. We define the \emph{induced bundle} (or \emph{pullback bundle}) $\pi_f : f^{\ast}P \to N$  by
	\begin{equation}\label{eq:f*P}
	f^{\ast}P := \{(x,p) \in N\times P : f(x) = \pi(p)\},
	\end{equation}
	with projection $\pi_f(x,p) := x$ and principal action
	\begin{equation}\label{niceaction} (x,p)s^{-1} := (x,sp).\end{equation}
	Standard arguments shows that $f^*P$ is a submanifold of $N\times P$ and $\pi_f$ is a smooth principal submersion. The $\star$-action can also be pulled back to $\pi_f$. Define
	\begin{equation}\label{eq:inducedstaraction}
	r(x,p) := (r\cdot x,gr^{-1}).
	\end{equation}
	Equation \eqref{eq:inducedstaraction} produces an action on $N\times P$ that leaves $f^{\ast}P$ invariant. $\pi_f$ is clearly equivariant with respect to \eqref{eq:inducedstaraction}.
	
	The induced bundle construction becomes a more interesting construction due to its relation with the original bundle $\pi:P\to N$.
	
	\begin{proposition}\label{prop:howtopullback} Let $f:N\to M$ be a smooth $G$-equivariant map and $P\stackrel{\pi}{\to} M$ the $G$-$G$-bundle associated to $\{\phi_{ij}:U_i\cap U_j\to G\}$. Then,
		\begin{enumerate}[$(i)$]
			\item $\pi_f:f^*P\to N$ is equivariantly isomorphic to  the $G$-$G$-bundle associated to the $\star$-collection $\{\phi_{ij} \circ f : f^{-1}(U_i) \cap f^{-1}(U_j) \to G\}$;
			\item the quotient $f^*P/\star$ is $G$-equivariantly diffeomorphic to  
			\[N' := \cup_{\widehat{\phi_{ij}\circ f}}f^{-1}(U_i);\]
			\item there is a well-defined map $f':N'\to M'$ such that, $f'|_{f^{-1}(U_i)}=f|_{f^{-1}(U_i)}$.
		\end{enumerate}
	\end{proposition}
	\begin{proof}
		The $G$-equivariance of $f$ guarantees the invariance of the sets $f^{-1}(U_i)$ and equivariance of $\phi_{ij}\circ f$:
		\[\phi_{ij}(f(g\cdot x))) = \phi_{ij}(g\cdot f(x)) = g\phi_{ij}(f(x))g^{-1}.\]
		To show that $f^*P$ is the bundle associated to $\{\phi_{ij}\circ f\}$, consider $P=\cup_{f_{\phi_{ij}}} U_i\times G$ and define a map $\Phi_i:\pi_f^{-1}(f^{-1}(U_i))\to f^{-1}(U_i)\times G$ via the expression
		\begin{equation}\label{proof:bijecao}
		(x,(f(x),g))\mapsto (x,g).
		\end{equation}
		The maps $\Phi_i$  clearly patch together to define a $G{\times }G$-equivariant diffeomorphism between $f^*P$ and $\cup_{f_{\phi_{ij}f}} (f^{-1}(U_i){\times }G)$. 
		
		Item $(ii)$  follows from  item $(i)$ and the second item in Theorem \ref{thm:star}.  For item $(iii)$, consider $f^*:f^*P\to P$, $f^*(x,p)=p$. $f^*$ is the map the makes the following \textit{pullback diagram} commutative:
		\begin{equation}
		\begin{xy}
		\xymatrix{f^*P\ar[r]^{f^*}\ar[d]^{\pi_f} & P\ar[d]^\pi \\ N\ar[r]^f & M }
		\end{xy}
		\end{equation}
		One observes that $f^*$ is $G\times G$-equivariant. In particular, it defines a map $f':f^*P/\star\to P/\star$. It follows from \eqref{proof:bijecao} that $f^*(\iota_i(x))=\iota_i(f(x))$, where $\iota_i$ is defined in \eqref{eq:iota}. Therefore, for $x\in f^{-1}(U_i)$,
		\[f'(x)=\pi(f^*(\iota_i(x)))=\pi(\iota_i(f(x)))=f(x).\qedhere\] 
	\end{proof}

	\section{Examples}
	\label{section:examples}
	
	%	The aim of this section is to construct  the examples in Theorem \ref{thm:main} through cross-diagrams whose bottom manifold ($M$ in \eqref{eq:CD}) has well-known geometry, such homogeneous spaces or sphere bundles.
	
	%	The examples are  mostly provide pull-back descriptions for two reasons: pull-back bundles inherit geometry from the pulled back bundle (for instance, the connection 1-form can also be pulled back); transition functions and trivializations are readily sorted out from the pullback description.
	
	In this section we provide basic examples of (special) $G$-$G$-bundles. 
	We start with the Hopf fibration $h:S^7\to S^4$ and linear $S^3$-bundles over $S^4$, followed by the bundles in \cite{speranca2016pulling} that realizes  homotopic 8-, 10- and Kervaire spheres (see also \cite{duran2001pointed,DPR,DRS}). 
	
	Examples \ref{ssec:Hopf} and \ref{ssec:GM} can be found in \cite{speranca2016pulling}. The bundles in Wilhelm \cite{wilhelm-lots} also can be described as $G$-$G$-bundles.
	
	\subsection{The Hopf $S^3$-$S^3$-bundle}\label{ssec:Hopf}
	Consider $S^7,S^4$ as  the unitary spheres on the quaternionic plane $\mathbb{H}\times\bb H$ and on $\bb R\times \bb H$, respectively. Define the \textit{Hopf map} $h:S^7\to S^4$ by
	\begin{equation}\label{eq:hopfmap}
	h\begin{pmatrix} x\\y\end{pmatrix}=\begin{pmatrix} |x|^2-|y|^2\\2x\bar y\end{pmatrix}.
	\end{equation}
	Let $S^3$ denote the unitary sphere on $\bb H$ and observe that, for $r,s\in S^3$,
	\begin{equation}\label{eq:hopfactions}
	h\begin{pmatrix}rx\bar s\\ry\bar s\end{pmatrix}=\begin{pmatrix}|x|^2-|y|^2\\r2x\bar y\bar r
	\end{pmatrix}.
	\end{equation}
	Condition \eqref{eq:isotropy} is easily verified (notice that it is sufficient to consider $x\in\bb R$). Therefore $h$ defines an $S^3$-$S^3$-bundle with $\star$-action defined by the $r$-multiplication. Since quaternionic conjugation on $S^7$ interchanges $r$ and $s$, the quotient $S^7/\star$ is diffeomorphic to $S^4$.

	Consider $D^4_\pm=\{(\lambda,x)\in S^4~|~\pm\lambda\geq 0\}$.  Local trivializations of $h$ are given by the maps $\Phi_\pm:D^4_\pm\to S^7$,
	\begin{equation}\label{}
	\Phi_+\begin{pmatrix}\begin{array}{c} \cos (t)\\\sin (t) \xi\end{array},~q \end{pmatrix}=\begin{pmatrix}\cos (t/2)\bar q\\\sin (t/2) \bar\xi\bar q \end{pmatrix},\quad \Phi_-\begin{pmatrix}\begin{array}{c} -\cos (t)\\\sin (t) \xi\end{array}, ~q \end{pmatrix}=\begin{pmatrix}\sin (t/2){\xi}\bar q\\\cos (t/2) \bar q \end{pmatrix},
	\end{equation}
	where $\xi$ is an unitary quaternion. Taking $U_0=D^4_+$ and $U_1=D^4_-$, we have $\phi_{01}:U_0\cap U_1=\{(0,\xi)\in S^4~|\xi\in \bb H\}\to S^3$ as $\phi_{01}(0,\xi)=\xi$. By identifying $U_0\cap U_1=S^3$ (dropping the first coordinate), we denote $\phi_{01}=I:S^3\to S^3$, the identity map.

	\subsection{The Gromoll-Meyer sphere}\label{ssec:GM}
	We recall the definition of the Lie group of quaternionic matrices $Sp(2)$
	\begin{equation}\label{eq:Sp2}
	Sp(2) = \left\{\begin{pmatrix} a & c \\b & d\end{pmatrix}\in S^7\times S^7~ \Big| ~\bar{c}a + \bar{d}b = 0\right\}.
	\end{equation} 
	The projection onto the first column $pr:Sp(2)\to S^7$ is a principal $S^3$-bundle with principal action:
	\begin{equation}\label{eq:GMprincipalaction}
	\begin{pmatrix} 
	a & c \\
	b & d 
	\end{pmatrix}\bar q = \begin{pmatrix}
	a & c\overline{q}\\
	b & d\overline{q}
	\end{pmatrix}.
	\end{equation}
	Gromoll and Meyer \cite{gromoll1974exotic} introduced the $\star$-action 
	\begin{equation}\label{eq:GMstaraction}
	q \begin{pmatrix} 
	a & c \\
	b & d 
	\end{pmatrix} = \begin{pmatrix} 
	qa\overline{q} & qc \\
	qb\overline{q} & qd 
	\end{pmatrix}.
	\end{equation}
	whose quotient is an exotic $7$-sphere, concluding their celebrated result on the existence of an exotic sphere with  non-negative sectional curvature:
	The corresponding action on $S^7$ can be ready from the first column of \eqref{eq:GMprincipalaction}:
	\begin{equation}\label{eq:GMactiondown}
	q\cdot\begin{pmatrix}x\\ y	\end{pmatrix}=\begin{pmatrix}qx\bar q\\ qy\bar q\end{pmatrix}
	\end{equation}
	The $S^3$-$S^3$-bundle defined by \eqref{eq:GMprincipalaction},\eqref{eq:GMstaraction} gives rise to the cross-diagram in Dur\'an \cite{duran2001pointed} which is used to geometrically produce an explicit clutching diffeomorphism $\hat b:S^6\to S^6$ for $\Sigma^7=Sp(2)/\star$:
	\begin{equation}\label{cd:duran}
	\begin{xy}\xymatrix{& S^3\ar@{.}[d]^{\bullet} & \\ S^3\ar@{..}[r]^{\star} &Sp(2)\ar[d]^{pr}\ar[r]^{pr'} &\Sigma^7\\ &S^7&}\end{xy}
	\end{equation}
	In order to produce $\hat b$, the geometry of $\Sigma^7$ is explored through \eqref{cd:duran}. This paper is dedicated to further advance the geometrical/topological relations started in \cite{duran2001pointed}.
	
	As an $S^3$-$S^3$-bundle, $pr$ can be realized as a pullback from $h:S^7\to S^4$. In \cite{speranca2016pulling}, this pullback realization is used to recover the identification of $Sp(2)/\star$ with the Milnor bundle $M_{1,-2}\cong M_{2,-1}$ (as in Gromoll--Meyer \cite{gromoll1974exotic}) which is  an exotic sphere.

	\subsection{Milnor bundles}\label{ssec:milnorbundles}
	The usual boundary map in the long homotopy sequence of the fibration $EG\to BG$, $G=SO(k)$ provides a bijection between the set o linear $S^{k-1}$-bundles over $S^n$ and $\pi_{n-1}SO(k)$. The linear $S^3$-bundles over $S^4$ are usually called Milnor bundles. As in Milnor \cite{mi}, observe that $t_{mn}:S^3\to SO(4)$,
	\begin{equation}
	t_{mn}(x)v=x^mvx^n, \quad v\in\bb H
	\end{equation}
	are representatives of $\pi_3SO(4)\cong \bb Z\oplus \bb Z$. We define $M_{m,n}=D^4\times S^3\cup_{f_{t_{m,n}}}D^4\times S^3$. Milnor observed that $M_{m,n}$ is homeomorphic to $S^7$ whenever $m+n=1$, but not diffeomorphic when $m=2$, for example (see \cite{ek} for a complete classification). On the other hand, the bundles $P_n=M_{0,n}\cong M_{-n,0}$ are $S^3$-principal. We use this section to show that every Milnor bundle can be obtained out of some $P_n$. 
	%	 \cite{davis}
	%	\footnote{Introduce Milnor bundles in general}
	%	A straightforward generalization of section \ref{ssec:GM} is in place: instead of $h:S^7\to S^4$ and $-h:S^7\to S^4$, one could 
	
	Consider a pair of $S^3$-principal bundles $\pi_k:P_k\to S^4$, $\pi_r:P_r\to S^4$. 
	%	Taking \[P_r=D^4\times S^3\bigcup_{f_{0,r}}D^4\times S^3\]	where $f_{m,n}(x,q)=(x,x^mqx^n)$ and, for simplicity, we use $D^4$ as the unitary disc in $\bb H$ instead of $D^4_\pm$. 
	$\pi_k$ is a $S^3$-$S^3$-bundle with the  $\star$-action
	\begin{equation}\label{eq:actionPk}
	r(x,q)=(rx\bar r,q\bar r).
	\end{equation}
	We consider $P_r$ as the $S^3$-manifold with action
	\begin{equation}\label{eq:actionPr}
	r\cdot (x,q)=(rx\bar r,rq\bar r).
	\end{equation}
	Observe that both \eqref{eq:actionPk} and \eqref{eq:actionPr} commutes with $f_{t_{n,0}}$ for every $n$, defining global actions on $P_n$. Given $\pi_k$, its unique trivialization function is $\phi_{01}(x)=x^k$, where both $U_0$ and $U_1$ are identified with $D^4$, the unit disc on $\bb H$. Consider the $S^3$-$S^3$-bundle given by $\pi_r^*P_k\to P_r$. Each copy $D^4\times S^3\subset P_r$ is invariant with respect to  \eqref{eq:actionPr}. The only transition function of $\pi^*_rP_k$ associated to the cover $\{D^4\times S^3, D^4\times S^3\}$ is $t_{k,0}\circ \pi_r|_{S^3\times S^3}:S^3\times S^3\to S^3$.. 
	
	From Theorem \ref{thm:star}, the resulting manifold is $M'=D^4\times S^3\cup_\psi D^4\times S^3$ where $\psi$ is the composition of $f_{0,r}$ with $(t_{k,0} \pi_r|_{S^3\times S^3})~\widehat{}~$. A straightforward computation gives $\psi=f_{t_{r,k-r}}$. We conclude that the bundle $M_{m,n}$ can be realized by a $S^3$-$S^3$-bundle $P_r\leftarrow P\to M_{m,n}$, where $r=m$.
	One sees that $r$ parametrizes the Euler class and $k$ the third homology of $M_{r,k-r}$ (see Milnor-Stasheff \cite{ms}). It is worth noticing that $(S^4)'=S^4$ in $S^4\leftarrow P_k\to (S^4)'$ and that the map $(\pi_r)':M_{r,k-r}\to S^4$ coincides with the bundle projection $\pi_{(r,k-r)}$.The full diagram is:
	
	%	 Considering $(i)$-$(iii)$ in Proposition \ref{prop:howtopullback}, one gets: 	
	\begin{equation}\label{cd:Milnorbundles}
	\begin{xy}\xymatrix@R=9pt@C=7pt@1{  & & S^3 \ar@{..}[dd] & & S^3 \ar@{..}[dd]& \\ &S^3\ar@{..}[rd] & & S^3\ar@{..}[rd]& &\\
		& & \pi_r^*P_k\ar[rd]\ar[rr]\ar[dd] & & P_k\ar[rd]\ar[dd]&\\ & & &M_{r,k-r}\ar@{-}[r] &{~}\ar[r] &S^{4}\\& & P_r\ar[rr]^{\pi_r} & & S^{4} &}\end{xy}
	\end{equation}
	The Gromoll--Meyer sphere happens with $r=1$, $k=-1$ (in this case, $P_1=P_{-1}=S^7$ and $\pi_{1}=-\pi_{-1}=h$ -- it is well known that the pull back of $h$ by $-h$ has total space $Sp(2)$, see \cite{speranca2016pulling,CR1}). 
	%	\begin{remark}
	%		Since \eqref{eq:actionPr} defines an action on any $M_{m,n}$, one can pull-back $\pi_k$ via  $\pi_{m,n}:M_{m,n}\to S^4$. The resulting manifold is diffeomorphic to $M_{m+r,n-r}$, therefore, the resulting classes reached by this procedure parametrizes, at most, the Euler class of the bundle. Since Morita equivalence is an equivalence relation between groupoids, it is natural to ask whether we can get a cross-diagram with both with any pair of Milnor bundles as base spaces.
	%	\end{remark}	
	
	\begin{remark}
		Another way to define a $S^3$-$S^3$-bundle over $P_r$ is to consider $P_r$ as the $S^3$-manifold with action \eqref{eq:actionPk}. A straightforward computation shows that the resulting diagram is $P_r\leftarrow \pi_r^*P_k\to P_{r-k}$. 
	\end{remark}
	
	\subsection{Other exotic spheres}
	
	In \cite{speranca2016pulling}, $G$-$G$-bundles were used to give geometric presentations of exotic spheres:
	
	\begin{theorem}[\cite{speranca2016pulling}, Theorems 1 and 2]\label{thm:8 10 Kervaire}
		There are explicit (special) $G$-$G$-bundles:
		\begin{enumerate}
			\item $S^8\stackrel{~\pi^{11}}{\longleftarrow}E^{11}\to \Sigma^8$, where $G=S^3$ and $\Sigma^8$ is the only exotic 8-sphere
			\item $S^{10}\stackrel{~\pi^{13}}{\longleftarrow}E^{13}\to \Sigma^{10}$, where $G=S^3$ and $\Sigma^{10}$ is a generator of the order 3 group of homotopy 10-spheres which bound spin manifolds.
			\item $S^{2n-1}\stackrel{~l_n}{\longleftarrow}L_{n}\to \Sigma^{2n-1}$, where $G=O(n)$ and $\Sigma^{2n-1}=\partial P^{2n}(\mathbf{A}^2)$
		\end{enumerate}
	\end{theorem}
	
	%	\footnote{More about $\theta^7,8,10$}
	%	$\Sigma^{2n-1}$ is homeomorphic to $S^{2n-1}$ whenever $n$ is odd, but only diffeomorphic for a finite set of $n$'s (see \cite{hill2016nonexistence,wang2016triviality} for recent advancements).
	%	
	$\pi^{11}$ and $\pi^{13}$ are pulled back from $pr:Sp(2)\to S^7$ and $l_n'$ from the frame bundle $pr_{n}:O(n{+}1)\to S^{n}$. As in the case of $pr$, we consider $O(n{+}1)$ as a matrix group and $pr_{n}$ as the projection to the matrix's first column. We give a brief description of  $\pi^{11},\pi^{13},l_n$.
	
	Consider $S^8 \subset \mathbb{R}\times \mathbb{H}\times\bb H$ endowed with the action 
	\begin{equation}\label{eq:action8}
	q\cdot \begin{pmatrix}
	\lambda \\
	x\\
	y
	\end{pmatrix} = \begin{pmatrix}
	\lambda\\ 
	qx\\ 
	qy\overline{q}
	\end{pmatrix}.
	\end{equation}
	Observe that $f_8 : S^8 \to S^7$, defined   by
	\begin{equation}
	f_8\begin{pmatrix}
	\lambda\\
	x\\ 
	w
	\end{pmatrix} = \dfrac{1}{\sqrt{\lambda^2 + |x|^4 + |w|^2}}\begin{pmatrix}\lambda + xi\overline{x}\\ w
	\end{pmatrix},
	\end{equation}
	is equivariant with respect to \eqref{eq:action8} and \eqref{eq:GMactiondown}. The bundle $\pi^{11}$ is defined as the  pullback of $pr:Sp(2)\to S^7$ by $f_8$. 
	
	We present $\pi^{10}$ with two different actions. Let $S^{10}\subset \rm{Im}\bb H\times \bb H\times \bb H$ be the $S^3$-manifold with one of the following actions\footnote{Although action (II) is not considered in \cite{speranca2016pulling}, going through the proof of Theorem 1 in \cite{speranca2016pulling}, one verifies that the sphere we get using action (II) is diffeomorphic to $\Sigma^{10}\#\Sigma^{10}\#\Sigma^{10}\#\Sigma^{10}$, which turns to be diffeomorphic to $\Sigma^{10}$.}
	\begin{equation}\label{eq:action10}
	\text{(I)}\quad	q\cdot \begin{pmatrix}
	p\\ 
	w\\ 
	x
	\end{pmatrix} = \begin{pmatrix}
	p\\ 
	qw\\ 
	qx\overline{q}
	\end{pmatrix}, \quad\qquad \text{(II)}\quad q\cdot \begin{pmatrix}
	p\\ 
	w\\ 
	x
	\end{pmatrix} = \begin{pmatrix}
	qp\bar q\\ 
	qw\bar q\\ 
	qx\bar{q}
	\end{pmatrix}.
	\end{equation}
	For the pulling back map, $f_{10}:S^{10}\to S^7$, define the \textit{Blakers-Massey element} $b: S^6 \to S^3$ as in \cite{ADPR,duran2001pointed}:
	\begin{equation}\label{eq:b}
	b(p,w) := \begin{cases}
	\frac{w}{|w|}\exp(\pi p)\frac{\overline{w}}{|w|}, ~\text{if}~w \neq 0,\\
	-1, ~\text{if}~w = 0.
	\end{cases}
	\end{equation}
	Let $\varphi:[0,1]\to[0,1]$ be a smooth non-decreasing function that is the identity on $[\frac{1}{4},\frac{3}{4}]$ and constant near 0 and 1, fixing 0 and 1. Set	
	\begin{align}\label{eq:f10}
	f_{10}\begin{pmatrix}\xi\\w\\x\end{pmatrix}=\begin{pmatrix}\sqrt{1-\varphi(|x|)^2}b\left(\frac{\xi}{\sqrt{|\xi|^2+|w|^2}},\frac{w}{\sqrt{|\xi|^2+|w|^2}}\right)\\\varphi(|x|)\frac{x}{|x|}\end{pmatrix}.
	\end{align}
	Note that $f_{10}$ is equivariant with respect to \eqref{eq:action10} and \eqref{eq:GMactiondown}. Define $\pi^{13}$ is the pullback of  $pr:Sp(2)\to S^7$  by $f_{10}$.
	\begin{remark}
		Action \eqref{eq:action10}-\emph{(I)}, can be extended to the effective $SO(4)=S^3\times S^3/\{\pm(1,1)\}$-action
		\begin{equation}\label{eq:action10-2}
		(q,r)\cdot \begin{pmatrix}
		p\\ 
		w\\ 
		x
		\end{pmatrix} = \begin{pmatrix}
		rp\bar r\\ 
		qw\bar r\\ 
		qx\overline{q}
		\end{pmatrix}
		\end{equation}
		which can be used to realize $\Sigma^{10}=(S^{10})'$ as a $SO(4)$-manifold. To this aim, one can either  observe that $f_{10}$ is invariant with respect to the $r$-coordinate (therefore the clutching diffeomorphism will have the desired equivariance) or consider $Sp(2)\times S^3/\{\pm(1,1)\}$ as an $SO(4)$-$SO(4)$-bundle with the $r$-coordinate acting only on the $S^3$-factor. According to a Conjecture 2 in Straume \cite{straume1994compact}, the $SO(4)$-manifold $\Sigma^{10}$ should be a peculiarly highly symmetric sphere among homotopy spheres that not bound parallelizable manifolds (see \cite{s1}).
	\end{remark}
	
	The $O(n)$-$O(n)$-bundle $\pi_n$ is realized as the pullback of $\rm{pr}_n:O(n{+}1)\to S^n$ by $J\tau_n$ (defined below). Observe that $\rm{pr}_n$ is an $O(n)$-$O(n)$-bundle with $\star$-action given by left multiplication of matrices: consider $O(n)$ acting through the morphism $s_n:O(n)\to O(n{+}1)$  and define $\rm{pr}_n(A)=Ae_0$, where $e_0=(1,0,0,...,0)^T$. The $O(n)$-action 
	\begin{equation}\label{eq:actionSn-Kervaire}
	r\star A=s_n(r)A
	\end{equation}
	is free and satisfies $\rm{pr}_n(rA)=s_n(r) \rm{pr}_n(A)$. An equivariant transition function for $\rm{pr}_n$ is $\tau_n:S^{n-1}\to O(n)$ defined as
	\begin{equation}
	\tau_n(x)v=2\lr{x,v}-v.
	\end{equation} 
	We consider $S^{2n-1}\subset \bb R^n\oplus \bb R^{n}$ and the map $J\tau_n:S^{2n-1}\to S^n$, 
	\begin{align*}
	J\tau_n(x_1,x_2)= \exp_{e_0}\left(\pi\tau_n\Big(\frac{x_2}{|x_2|}\Big)x_1\right).
	\end{align*}
	$J\tau_n$ is equivariant with respect to \eqref{eq:actionSn-Kervaire} and the \textit{biaxial action}
	\begin{equation}\label{eq:actionS2n-1-Kervaire}
	r\cdot\begin{pmatrix}x\\ y\end{pmatrix}=\begin{pmatrix}rx\\ ry\end{pmatrix}.
	\end{equation}
	
	The resulting $O(n)$-$O(n)$-bundle $l_n{:}\,L_n=J\tau_n^*O(n{+}1)\to S^{2n-1}$ gives rise to the manifold $\Sigma^{2n-1}=L_n/\star$. 
	Following \cite{speranca2016pulling}, $\Sigma^{2n-1}$ is homeomorphic to $S^{2n-1}$ for $n$ odd and has $H_{n-1}(\Sigma^{2n-1})\cong \bb Z_3$ when  $n$ is even (as mentioned in the introduction,  $\Sigma^{2n-1}=\partial P^{2n}(A_2)$ in Bredon \cite[Chapter V.9]{brebook}). Next we highlight two phenomena depending on the parity of $n$.
	
	If $n=2m$, $O(2m)$ admits a subgroup $S^1\subset O(2m)$ acting freely on $S^{2n-1}$: consider $S^1$ as the subgroup of block diagonal matrices $diag(A,A,...,A)$, where $A\in SO(2)$ (i.e.,  the exponential of the standard  complex structure of $\bb R^{2m}$). From Theorem \ref{thm:star}, item $(ii)$, the resulting  $S^1$-action is free on $\Sigma^{2n-1}$. The resulting quotient $\Sigma^{4m-1}/S^1$, denoted by $\Sigma P\bb C^{2m-1}$, has $\pi_k \Sigma P\bb C^{2m-1}\cong \pi_k\bb CP^{2m-1}$ for $k\neq 2m$ and $\pi_{2m}\Sigma P\bb C^{2m-1}\cong \bb Z_3$. 
	
	When $n=4m$, one gets a free $S^3$-action by considering $A\in Sp(1)\subset O(4)$ in $diag(A,...,A)$. The resulting quotient ,$\Sigma^{8m-1}/S^3:=\Sigma P\bb H^{2m-1}$, satisfies $\pi_k \Sigma P\bb H^{2m-1}\cong \pi_k\bb HP^{2m-1}$ for $k\neq 4m$ and $\pi_{4m}\Sigma P\bb H^{2m-1}\cong \bb Z_3$. We also observe that there are subgroups $U(m)\subset O(2m)$, $Sp(m)\subset O(4m)$ commuting with the $S^1$-, respectively $S^3$-action, defining $\Sigma P\bb C^{2m-1}$ as a $U(m)$-manifold and $\Sigma P\bb H^{2m-1}$ as a $Sp(m)$-manifold. In section \ref{sec:geometry}, we provide invariant metrics of positive Ricci curvature on both $\Sigma P\bb C^{2m-1}$ and  $\Sigma P\bb H^{2m{-}1}$.
	
	If $n$ is odd, $n=2m{+}1$, we  consider the bundle reduction $\rm{pr}_m^{\bb C}:U(m{+}1)\to S^{2m{+}1}$. Considering $U(m{+}1)\subset O(n+1)$, we take $\rm{pr}_m^\bb C$ as the restriction $\rm{pr}_n|_{U(m{+}1)}$. By observing that both right and left multiplication by $U(m)\subset O(n)$ leaves $U(m{+}1)$ invariant, one concludes that $\rm{pr}_m^\bb C$ is a $U(m)$-$U(m)$-bundle. An equivariant transition map $\tau_m^\bb C:m:S^{2m}\to U(m)$ is presented in \cite{puttmann2003presentations}:
	
	\begin{equation}\label{eq:tauC}
	\tau_m^\bb C\begin{pmatrix}y\\  z\end{pmatrix}= \left(\id -\frac{z}{|z|}(1+e^{\pi y})\frac{\bar z^t}{|z|}\right),
	\end{equation}	
	where $y\in i\bb R$ and $z\in\bb C^m$. 	We can therefore consider A $U(m)$-reduction of $L_n$ can be realized by the pull back of $J\tau^\bb C_m:S^{4m+1}\to S^{2m+1}$,
	\begin{align*}
	J\tau^\bb C_m(x_1,x_2)= \exp\left(\pi\tau^\bb C_m\Big(\frac{x_2}{|x_2|}\Big)x_1\right).
	\end{align*}
	$J\eta_m$ is equivariant with respect to \eqref{eq:actionSn-Kervaire} and the $U(m)$-action defined by restricting  \eqref{eq:actionS2n-1-Kervaire} to $U(m)\subset O(n-1)$. 
	
	If $n=4m+3$, the analogous reduction $Sp(m+1)\subset O(n+1)$ works along the same lines, with transition map $\tau^{\bb H}_m:S^{4m+2}\to Sp(m)$, obtained by replacing $i\bb R$ by $\rm{Im}\bb H$ and $\bb C^m$ by $\bb H^m$ in \eqref{eq:tauC}. 
	
	The advantage of the $U(m)$, $Sp(m)$ realization of $\Sigma^{4m+1}$ is the presence of fixed points, allowing us to perform equivariant connected sums (see section \ref{ssec:connectedsum}).
	
	\section{Connected Sums}\label{ssec:connectedsum}
	Here we realize  diagrams of the form $M\leftarrow f^*P\to M\#\Sigma^n$, where $\Sigma^n$ is an exotic sphere. In  section \ref{sec:geometry}, the diagram is used to study the Ricci curvature of $M\#\Sigma^n$  avoiding the intricate geometry of a connected sum.
	
	%	Obviously, there are more general and direct methods to realize a metric in $M\# \Sigma^n$, however
	
	Let $M$ be a $G$-manifold with a fixed point $p\in M$. Assume  $G$ is compact and that $M$ has a $G$-invariant Riemannian metric. The differential of the action at $p$ induces a morphism $\rho:G\to O(T_pM)$ called the \textit{isotropic representation of $G$ at $p$}. 
	
	On the other hand, we consider $S^n\subset \bb R{\times} T_pM$ as a $G$-manifold with action $g\cdot (\lambda,x)=(\lambda,\rho(g)x)$. Note that  $e_0=(1,0)^T$ is a fixed point. Consider $D_\pm=\{(\lambda, x)\in S^n~|\pm\lambda \geq 0\}$ and $S^{n-1}=D_+\cap D_-$, observing that $D_+,D_-,S^{n-1}$ are $G$-invariant subsets. If $\phi:S^{n-1}\to G$ satisfies the equivariant  condition \eqref{eq:conjugation}, one constructs the $G$-$G$-bundle $P\to S^n$ 
	\begin{equation}\label{proof:connectedsumP}
	P=D_+\times G\bigcup_{f_\phi} D_-\times G,
	\end{equation}
	where $f_\phi(x,g)=(x,g\phi(x))$ and 
	%	the non-trivial relations of $\sim$ are  $D_+\times G\ni (x,g)\sim (x',g')\in D_-\times G$ if and only if $x,x'\in D_+\cap D_-$ and $g'=g\phi(x)$. 
	the $\star$-action is defined as in Theorem \ref{thm:star}: $r(x,g)=(r\cdot x,gr^{-1})$.
	Using \eqref{proof:connectedsumP}, Theorem \ref{thm:star} identifies  $(S^n)'$ as the {twisted sphere}   $(S^n)'=D^n\cup_{\hat{\phi}} D^n$. 
	
	As a next step, we pull back $P$ to $M$. The pullback function $f:M\to S^n$ can be obtained along the lines of the Thom--Pontrjagyn construction (see Kosinski \cite{ko}, sections IX.4 and IX.5): let $D_\epsilon$ be an $\epsilon$-disc around the fixed point $p$ such that $\exp_p|_{D_\epsilon}$ is a diffeomorphism. Define $f:M\to S^n$ as
	\begin{equation}\label{eq:fconnected}
	f(x)=\begin{cases} \exp_{e_0}\Big({\pi}\varphi\Big(\frac{|\exp_p^{-1}(x)|}{{\epsilon}}\Big)\frac{\exp_p^{-1}(x)}{|\exp_p^{-1}(x)|}\Big), ~&x\in D_\epsilon\\ -e_0, & x\notin D_\epsilon\,.\end{cases}
	\end{equation} 
	
	As $P$ is trivial along $D_\pm$, Proposition \ref{prop:howtopullback}, item $(i)$, gives
	\begin{equation}
	f^*P=f^{-1}(D_+)\times G\bigcup_{f_{\phi\circ f}} f^{-1}(D_-)\times G=(M-D_{\epsilon/2})\times G\bigcup_{f_{\phi}} D_{\epsilon/2}\times G
	\end{equation}
	where $D_{\epsilon/2}$ is the $\epsilon/2$-disc around $p$ and we identify $S^{n-1}$ with $\partial D_{\epsilon/2}$. Therefore,
	\begin{equation}
	M'=(M-D_{\epsilon/2})\bigcup_{\hat{\phi}} D_{\epsilon/2}.
	\end{equation}
	$M'$ is easily seem to be diffeomorphic to $M\# (S^n)'$. We have:
	
	\begin{theorem}\label{thm:connectedsum}
		Let $M^n$ be  $G$-manifolds and $p\in M$  with isotropic representation $\rho:G\to O(n)$. Then, given a map $\phi:S^{n-1}\to G$ satisfying 
		\[\phi(\rho(g)x)=g\phi(x)g^{-1}, \]
		there is a $G$-$G$-bundle $M\leftarrow P\to M\#\Sigma^n$, where $\Sigma^n=D^n\cup_{\hat \phi}D^n$.
	\end{theorem}
	
	\begin{remark}Following Lemma \ref{lem:reentrance}, of $\phi$ satisfies the hypothesis on Theorem \ref{thm:connectedsum}, so does $\phi^k$, for every $k\in \bb Z$. Since $\#_k\Sigma^{n}=D^n\cup_{\widehat{\phi}^k}D^n$, the resulting $M'$ gives the $k$-fold connected sum $M\#\Sigma^n\#...\#\Sigma^n$.
	\end{remark}	
	\begin{remark}
		One can use the Thom--Pontrjagyn to provide a more general construction: suppose $N^k\subset M^{n+k}$ is a $G$-invariant manifold and $F:\bb R^n\times N\to M$ is a frame of $\nu N$ such that $G$ satisfies $gF(v,x)=F(\rho(g)v,gx)$ for some linear representation $\rho:G\to O(n)$. One considers the  $f:M\to S^{n+1}$,
		\begin{equation}\label{eq:fpair}
		f(z)=\begin{cases} \exp_{e_0}\Big({\pi}\varphi(|v|)\frac{v}{|v|}\Big), ~&z=F(x,v)\\ -e_0, & z\notin F(\bb R^k\times N)\,.\end{cases}
		\end{equation} 
		Given a function $\phi: S^{n-1}\to G$ as in Theorem \ref{thm:connectedsum}, one can again consider the $G$-$G$-bundle \eqref{proof:connectedsumP} and its pull-back $f^*P$. The resulting manifolds, $M'$, resembles Bredon's pairing \cite{bre1} as a version of $M$ `twisted' along the boundary of a tube around $N$. Examples based on this construction will be presented elsewhere.
		%		Bredon \cite{bre1} considers the principal orbit admits a frame $F$ as above. In this case, the quotient space $M'/G$ can be somehow identifies with $M/G\#\Sigma^n$ wince the connected sum is done in the manifold part of $M/G$. On the other hand, we shall see that, whenever $G$ is comapact, $M'$ admits a $G$-invariant metric such that $M/G$ is isometric to $M'/G$.
	\end{remark}

	It follows  that  there are functions $\phi:S^{n-1}\to G$ realizing  $P=D^n\times G\cup_{f_\phi} D^n\times G$ and, therefore, $\Sigma^n=D^n\cup _{\widehat \phi}D^n$, where  $P$ is $E^{11},~E^{13}$ and $L_{2m+1}$ in Theorem \ref{thm:8 10 Kervaire} (Proposition \ref{prop:fibradonormal} or Theorem 4.2 and Proposition 4.4 in \cite{speranca2016pulling}). Explicit formulas for $\phi$ in the case of $E^{11}$ and $E^{13}$ are presented in \cite[Theorem 4.6]{speranca2016pulling}. For $b:S^6\to S^3$ in \ref{eq:b}, $\Sigma^7=D^7\cup_{\hat b} D^7$ is a generator of the group of homotopy 7-spheres (\cite{ADPR,duran2001pointed}). 
	
	Observe that the representations $(\rho_n)$ in the introduction are related to the fixed points of the $S^3$-manifolds $\Sigma^7,\Sigma^8,\Sigma^{10}$, the $U(m)$-manifold $\Sigma^{4m+1}$ and the $Sp(m)$-manifold $\Sigma^{8m+5}$, respectively. To realize the examples in Theorem \ref{thm:main} we just need to observe that $M^n$  in Theorem \ref{thm:main} has a $G$-action carrying a fixed point with isotropy representation $(\rho_n)$.
	
	%	Let us assume (postponing the proof to section \ref{sec:geometry}) that our examples of exotic spheres  are provided by a $G$-$G$-bundle as in \eqref{proof:connectedsumP} ($M_{1+k,-k}$, $\Sigma^8$, $\Sigma^{10}$, the $U(m)$-manifolds $\Sigma^{4m+1}$ and the $Sp(m)$ $\Sigma^{8m+5}$ -- we exclude the $O(n)$-manifolds $\Sigma^{2n-1}$ since its standard counterparts', $S^{2n-1}$, do not have fixed points). We seek $G$-manifolds $M$ for which we can perform our geometric-induced connected sum with such spheres. Therefore, we seek manifolds with isotropic representations relative to \eqref{eq:GMactiondown}, \eqref{eq:action8}, \eqref{eq:action10}, \eqref{eq:actionS2n-1-Kervaire}. That is:
	%	
	%	\begin{enumerate}[$(a)$]
	%		\item $n=7$, $G=S^3$: $\rho=\rho_{1}\oplus \rho_1\oplus \rho_0$, where $\rho_1:S^3\to O(3)$ is the \textit{spin 1}  representation induced by the double cover $S^3\to SO(3)$ and $\rho_0:S^3\to O(1)$ is the trivial one 
	%		\item $n=8$, $G=S^3$: $\rho=\rho_{\h}\oplus \rho_1\oplus \rho_0$, where $\rho_{\h}:S^3\to O(4)$ is the \textit{spin $\h$} representation induced by right (or left) multiplication by a quaternion
	%		\item $n=10$, $G=S^3$: $\rho=\rho_{\h}\oplus \rho_1\oplus 3\rho_0$ or $\rho=\rho_1\oplus \rho_1\oplus \rho_1\oplus \rho_0$
	%		\item $n=4m+1$, $G=U(m)$: $\rho=\rho_{U(m)}\oplus\rho_{U(m)}\oplus\rho_0$ where $\rho_{U(m)}:U(m)\to O(2m)$ is the standard representation
	%		\item $n=8m+5$, $G=Sp(m)$: $\rho=\rho_{Sp(m)}\oplus\rho_{Sp(m)}\oplus 5\rho_0$ where $\rho_{Sp(m)}:Sp(m)\to O(4m)$ is the standard representation
	%	\end{enumerate}
	
	The reducibility of the representations $(\rho_7)$-$(\rho_{8m+5})$ allows us to explore $G$-manifolds which are locally products, such as (trivial and non-trivial) bundles. In what follows, we list some examples. As `building blocks', we use  homogeneous manifolds such as $S^n$  and projective spaces $\bb CP^m,\bb HP^m$. If $H/K$ is an homogeneous manifold, the pair $(H/K,\rho')$ is to be interpreted as the manifold $H/K$ endowed with the $G$-action  induced by $\rho'(G)\subset K$. Thus, $\rho'$ is the isotropy representation of the $G$-manifold $H/K$ at the `{base point}'  $K\in H/K$. In general, $(N,\rho')$ will denote a manifold that posses a fixed-point with isotropic representation $\rho'$. For the representations, we use the notation in the list $(\rho_7)$-$(\rho_{8m+5})$.
	
	\subsubsection{Product manifolds}
	\begin{enumerate}
		\item $(S^3,\rho_1)\times (S^4,\rho_1\oplus \rho_0)$
		\item $(S^3,\rho_1)\times (S^3,\rho_1)\times (N^1,\rho_0)$
		\item $(S^6,\rho_1\oplus \rho_1)\times (N^1,\rho_0)$
		\item $(S^3,\rho_1)\times (S^5,\rho_{\h}\oplus \rho_0)$
		\item $(S^3,\rho_1)\times (S^4,\rho_{\h})\times (N^1, \rho_0)$
		\item $(S^7,\rho_\h\oplus \rho_1)\times (N^1,\rho_0)$
		\item $(N^8,\rho_\h\oplus \rho_1\oplus\rho_0)\times (N^2,2\rho_0)$
		\item $(S^{2m},\rho_{U(m)})\times (S^{2m+1},\rho_{U(m)}\oplus\rho_0)$
		\item $(S^{4m},\rho_{U(m)}\oplus\rho_{U(m)})\times (N^1,\rho_0)$
		\item $(\bb CP^m,\rho_{U(m)})\times (S^{2m+1},\rho_{U(m)}\oplus\rho_0)$ 
		\item $(\bb CP^m,\rho_{U(m)})\times (S^{2m},\rho_{U(m)})\times (N^1,\rho_0)$ 
		\item $(\bb CP^m,\rho_{U(m)})\times (\bb CP^{m},\rho_{U(m)})\times (N^1,\rho_0)$ 
		\item $\left(\frac{U(m+2)}{SU(2)\times U(m)},\rho_{U(m)}\oplus\rho_{U(m)}\oplus\rho_0\right)$
		\item $\left(\frac{U(m+2)}{U(2)\times U(m)},\rho_{U(m)}\oplus\rho_{U(m)}\right)\times (N^1,\rho_0)$
		\item $(S^{4m},\rho_{Sp(m)})\times (S^{4m+5},\rho_{Sp(m)}\oplus 5\rho_0)$
		\item $(S^{8m},\rho_{Sp(m)}\oplus\rho_{Sp(m)})\times (N^5,5\rho_0)$
		\item $(\bb HP^m,\rho_{Sp(m)})\times (S^{4m+l},\rho_{Sp(m)}\oplus l\rho_0)\times (N^{5-l},(5-l)\rho_0)$, for $0\leq l\leq 5$ 
		\item $(\bb HP^m,\rho_{Sp(m)})\times (\bb HP^{m},\rho_{Sp(m)})\times (N^5,5\rho_0)$ 
		\item $\left(\frac{Sp(m+2)}{Sp(2)\times Sp(m)},\rho_{Sp(m)}\oplus\rho_{Sp(m)}\right)\times (N^5,5\rho_0)$
	\end{enumerate}	
	
	\begin{remark}
		As mentioned, $T_1S^{2m+1}\# \Sigma^{4m+1}\cong T_1S^{2m+1}$ (see de Sapio \cite{DeS}).
	\end{remark}
	
	\subsection{Bundles over spheres}
	To consider actions on bundles, we rely on explicit equivariant expressions of transition functions, such as \eqref{eq:b}.
	
	\subsubsection{Milnor bundles} A family of examples are given by the Milnor bundles:  action \eqref{eq:actionPr}  is well defined in every $M_{m,n}$ and  $(0,1)^T$ is a fixed point with isotropy representation $(\rho_7)$.
	
	\subsubsection{Explicit non-linear $S^6$-bundles over $S^1$} Given a diffeomorphism $h:S^6\to S^6$, a $S^6$-bundles over $S^1$ is defined by
	\begin{equation}
	E_h=[0,1]\times S^6\cup_{f_h}[0,1]\times S^6
	\end{equation}
	where $f_h:\{0,1\}\times S^6\to \{0,1\}\times S^6$ is defined by $f_h(0,x)=(0,x)$, $f_h(1,x)=(1,h(x))$. There are interesting choices we can make for $h$: by taking $h=\alpha$, the antipodal map in $S^6$ one gets the non-trivial linear $S^6$-bundle  over $S^1$. Moreover, \cite{ADPR} provides the family $S^3$-equivariant  fixed-point free involutions $\theta_k=\alpha \hat b^k:S^6\to S^6$. It is also proved in \cite{ADPR} that $\alpha=\theta_0,...\theta_{27}$ parametrizes the 28 connected components of $\rm{Diff}^-(S^6)$, the set of orientation-reversing diffeomorphisms of $S^6$. 	Therefore, $E_k=E_{\theta_k}$ is isomorphic to a linear bundle if and only $k\equiv0\mod 28$. 
	
	One can define $E_k$ as an $S^3$-manifold using action \eqref{eq:GMactiondown} on $S^6$ (taking $S^6=S^7\cap \{\Re(x)=0\}$). It has a fixed point ($(0,(0,1)^T)$ in any copy of $[0,1]\times S^6$) with isotropy representation $(\rho_7)$. The authors don't known whenever $E_k$ is diffeomorphic to $E_0$ or not.
	
	The advantage of using $\theta_k$ instead of the orientation-preserving $\hat b^k$ is that $\theta_k$, being a fixed-point free involution, defines a free $\bb Z_2$-action on $E_k$. The quotient is the product of a \textit{fake projective plane} $F\bb RP^6$ and $S^1$.
	
	Although the observations above does provide $S^3$-$S^3$-bundles $E_k\leftarrow P\to E_k\#\Sigma^7$, we are not able to give any new relevant information about $E_k\#\Sigma^7$ since the geometry of $E_k$ is itself unknown to the authors (for instance, $E_k$ can't have positive Ricci curvature since it has infinite fundamental group). One might believe that $E_k\cong E_0\#_k\Sigma^7$. In this case, there is a cross-diagram $E_0\leftarrow P\to E_k$ and, possibly, something could be said about $E_k$.
	
	%		\begin{remark}
	%			Although the examples involving $S^1$ does not admit a metric with positive Ricci curvature, the present method allows us to distinguish a one-dimensional foliation on $(S^1\times S^6)\# \Sigma^7$ whose (local) quotient has positive Ricci curvature. Despise the integral submanifolds of the foliation be the natural replacement for the submanifolds $\{*\}\times S^6$, it is not clear if they are closed curves or not.
	%		\end{remark}
	
	%for $\Sigma_{\theta_k}^7=D^7\cup_{\theta_k}D^7$, $\Sigma^7_{\theta_0},...,\Sigma^7_{\theta_{27}}$ parametrizes $\theta^7\approx\bb Z_{28}$.
	
	%	\begin{remark}
	%		Observing that the support of $\hat b^k$ (i.e., the points where $\hat b^k$ is not the identity) can be made to fit a disc on $S^6$, some identity such as $E_k\#\Sigma_{\theta_r}=E_{k+r}$ might hold.
	%	\end{remark}
	
	%	\begin{remark}
	%		Sadly, our geometric description of the resulting connected sums are, at least, as complicated as the geometric description of the original manifold. Therefore, we do not add any new geometric information to such interesting examples.
	%	\end{remark}
	%	\textbf{Non-trivial $S^6$-bundle over $S^1$}
	
	%	Moreover,  \eqref{eq:GMactiondown} commutes with the $SO(2)$-action generated by the endomorphism
	%	\begin{equation}
	%	J\begin{pmatrix}x\\ y	\end{pmatrix}=\begin{pmatrix}\Re(x)-\rm{Im}(y)\\\Re(y)+\rm{Im}(x)\end{pmatrix}.
	%	\end{equation}
	%	Since the image of $\exp(tJ)$ represents a non-trivial element in  Using the $SO(2)$-subgroups generated
	
	%	\subsubsection{8-dimensional examples}
	\subsection{8-dimensional bundles}
	Linear $S^3$-bundles over $S^5$ and $S^4$-bundles over $S^4$ are parametrized by $\pi_4SO(4)\approx\bb Z_2+\bb Z_2$ and $\pi_3SO(5)\approx\bb Z$, respectively. Define:
	\begin{align}
	\eta:S^4&\to S^3 \\
	\begin{pmatrix}\lambda\\ x \end{pmatrix} &\mapsto \frac{\lambda +xi\bar x }{|\lambda +xi\bar x |}\nonumber
	\end{align}
	Then $\eta_L=\rho_L\circ\eta$ and $\eta_R=\rho_R\circ \eta$ are generators for $\pi_4SO(4)$, where $\rho_L,\rho_R$ denotes the 4-dimensional representation of $S^3$ defined by left, respectively right, multiplication by quaternions. $\rho_1\oplus 2\rho_0=I_5:S^3\to SO(5)$ generates $\pi_3SO(5)$. Note that  $\eta_L,\eta_R$  are equivariant with respect to $\rho_1\oplus 2\rho_0$ and $I_5(qx\bar q)=I_5(q)I_5(x)I_5(q)^{-1}$. We get the bundles
	\begin{equation}
	P_{\eta_\epsilon}=D^5\times S^3\bigcup_{f_{\eta_\epsilon}}D^5\times S^3,\qquad P_{I_5^k}=D^4\times S^4\bigcup_{f_{I_5^k}}D^4\times S^4,
	\end{equation}
	where $\epsilon=L,R$ and $I_5^k(q)=I_5(q)^k$. Both bundles admit the actions 
	\begin{equation}
	q\cdot \left(\begin{array}{c}\lambda \\ x\end{array}, g\right)=\left(\begin{array}{c}\lambda \\ qx\end{array}, qg\bar q\right),\qquad q\cdot \left(\begin{array}{c}\lambda_1 \\ x_1\end{array},\begin{array}{c}\lambda_2 \\ x_2\end{array} \right)=\left(\begin{array}{c}\lambda_1 \\ qx_1\bar q\end{array}, \begin{array}{c}\lambda_2 \\ qx_2\end{array}\right).
	\end{equation}
	Therefore $P_{\eta_\epsilon},P_{I_5^k}$ have fixed points with isotropy representation $(\rho_8)$.
	
	\begin{remark}
		The analogous bundle $P_{\eta_L\eta_R}$, whose transition function is given by the product $\eta_L\eta_R(z)=\eta_L(z)\eta_R(z)$ and $P_{I_5^k}$, $k\neq 0\mod 2$ are stabilized by $\Sigma^8$, i.e., $M\#\Sigma^8\cong M$, for $M=P_{\eta_L\eta_R},P_{I^k_5}$ (see de Sapio \cite{de1969manifolds} -- recall that $\Sigma^8=\sigma_{4,3}(\eta_L\eta_R,I_5)$ and that $\theta^8\approx \bb Z_2$). If $k$ is even, $P_{I_5^k}\#\Sigma^8\ncong P_{I_5^k}$ (see de Sapio \cite{de1969manifolds}). The authors does not known whenever $P_{\eta_\epsilon}$ is stabilized or not by $\Sigma^8$.
	\end{remark}
	
	\begin{remark}
		Another $S^3$-manifold with istropy  $(\rho_8)$ is $\bb HP^2$ with an $S^3=Sp(1)$-action derived from a suitable subgroup of $Sp(2)\times Sp(1)$,  its standard isotropic representation. However, the promising manifold $\bb HP^2\#\Sigma^8$ is diffeomorphic to $\bb HP^2$ (as pointed out to the first author by D. Crowley -- see Kramer--Stolz \cite{kramer2007} for a reference). 
	\end{remark}
	
	\subsubsection{10 dimensional bundles}
	The case of $S^3$-bundles over $S^7$ is of special interest. We deal with the principal case for simplicity. It is known that $b:S^6\to S^3$, defined in \eqref{eq:b}, is a generator of $\pi_6S^3$ (it is obtained as an explicit \textit{clutching map} of $pr:Sp(2)\to S^7$ in \cite{duran2001pointed}). It admits the following $SO(4)$ symmetry:
	\begin{equation}
	b\hspace{-1.5pt}\begin{pmatrix}rp\bar r\\ qw\bar r\end{pmatrix}=qb\hspace{-1.5pt}\begin{pmatrix}p\\ w\end{pmatrix}\bar q.
	\end{equation}
	Therefore, the map $f_b:S^6\times S^3\to S^6\times S^3$ is equivariant with respect to the $(S^3)^3$-action:
	\begin{equation}
	(q,r,s)\cdot \left(\begin{array}{c}p\\w \end{array}, g\right)=\left(\begin{array}{c}rp\bar r\\qw\bar r \end{array}, sg\bar q\right),
	\end{equation}
	defining $P_{b^k}=D^7\times S^3\cup_{f_{b^k}}D^7\times S^3$ as a $(S^3)^3$-manifold. $P_{b^k}$ satisfies $(c)$ by restricting the action to the diagonal $\{(r,r,r)~|~r\in S^3\}$. A generic linear $S^3$-bundles is obtained using the transition map $L_{b^n}\circ R_{b^m}$. Denote this bundle by $P^{10}_{m,n}\to S^7$. Since $f_{L_{b^n}\circ R_{b^m}}$ is still equivariant with respect to the diagonal action,  $P^{10}_{m,n}$ satisfies $(c)$.
	
	\begin{remark}
		Recalling from \cite{speranca2016pulling} that $\Sigma^{10}=\sigma_{3,6}(I_6,L_b\circ R_b)$, one concludes that $P_{b^k}\#\Sigma^{10}\cong P_{b^k}$ if and only if $k\neq 0\mod 3$ (see de Sapio \cite[Theorem 1]{de1969manifolds}). The authors were not able to find out whether $P_{m,n}^{10}\#\Sigma^{10}$ is diffeomorphic to $P_{m,n}^{10}$ or not. 
	\end{remark}
	\begin{remark}
		The diffeomorphism $\widehat{\Theta}_{10}$ in \cite[Theorem 4.6]{speranca2012Phd} provides an explicit representative of a class in $\pi_3(\rm{Diff}^+(S^6))$ not represented by $\pi_3(O(7))$. More specifically, it furnishes a non-linear $S^6$-bundle over $S^4$ whose total space is an $S^3$-manifold with a fixed point with isotropy $(\rho_{10})$. Note that the works of Nash \cite{nash1979positive} or Poor \cite{poor1975some} does not provide positive Ricci on this bundle (supposing its total space is not diffeomorphic to a known space).
	\end{remark}

	\section{The geometry of cross-diagrams}
	\label{sec:geometry}
	
	Let $M\stackrel{\pi}{\leftarrow }P\stackrel{\pi'}{\to} M'$ be a $G$-$G$-bundle. An efficient way to compare geometries in $M$ and $M'$ is to endow $P$ with a $G\times G$-invariant metric. In this case, the space $\cal H''\subset TP$, orthogonal to the $G\times G$-orbits on $P$, descends isometrically to both $\cal H$ and $\cal H'$, the spaces orthogonal to the $G$-orbits on $M$ and $M'$, respectively.
	
	In what follows, we provide more details about existence of $G\times G$-invariant metrics on $P$ and explore the transversal geometry induced by the isometries $\cal H\stackrel{d\pi}{\leftarrow}\cal H''\stackrel{d\pi'}{\to}\cal H'$. We suppose  $G$ compact with a bi-invariant metric originating from an adjoint invariant inner product $Q$ on $\lie g$. 
	
	There are three different ways to endow $P$ with a $G\times G$-invariant metric:
	\begin{enumerate}
		\item Since $G\times G$ is compact, averaging any initial metric $g_0$ on $P$ gives an invariant metric $g$ (see Bredon \cite{brebook} for reference)
		\item A more concrete metric can be obtained our examples: given a map $f:N\to M$, $f^*P$ is naturally a submanifold of $N\times P$ (equation \eqref{eq:f*P}). If $M$ and $P$ are equipped with a $G$-invariant and a $G\times G$-invariant metric, respectively, the induced metric on $f^*P$ is $G\times G$-invariant. Most of our examples are the pullback of either $pr:Sp(2)\to S^7$ or the frame bundle $pr_n:O(n+1)\to S^n$, which admit natural $G\times G$-invariant metrics
		\item Given a connection 1-form $\omega:TP\to \lie g$ and a metric $g_M$ on $M$, one can endow $P$ with the \emph{Kaluza-Klein metric} 
		\begin{equation}\label{eq:kk}
		\langle X, Y\rangle := g_M(d\pi X, d\pi Y) + Q(\omega(X),\omega(Y)). 
		\end{equation}		If both $g_M$ and $\omega$ are $G$-invariant (in the sense that $\omega_{gp}(gX)=\omega_p(X)$ for all $X\in TP$) then, $\lr{,}$ is $G\times G$-invariant
	\end{enumerate}
	
	We focus on $(3)$ and prove:
	\begin{proposition}\label{prop:existence-invariant-metric}
		There exists a connection 1-form $\omega:TP\to \lie g$ such that
		\[\omega_{rp}(rX)=\omega_p(X)\]
		for all $X\in TP$ and $r\in G$. Moreover, if $g_M$ is a $G$-invariant metric, then the metric \eqref{eq:kk} is $G\times G$-invariant. 
	\end{proposition}
	
	It follows immediately:
	
	\begin{corollary}
		Let $M\leftarrow P\to M'$ be a $G$-$G$-bundle. Then there is a $G$-invariant metric on $M'$ such that $M/G$ and $M'/G$ are isometric as metric spaces.
	\end{corollary}
	
	%	From now one, let $P$ be equipped with a $G{\times} G$-invariant metric and assume $\pi,\pi'$ Riemannian submersions.  All our geometric comparison are related to the underlaying fact that $M/G$ and $M'/G$ are isometric (observe that both spaces are isometric to $P/(G\times G)$). The next result compares focal points between orbits in $M$ and $M'$.
	
	First observe there are bijections between the set of orbits of $M$, $M'$ and $P$: if $x=\pi(p)$ and $x'=\pi'(p)$, then $\pi^{-1}(Gx)=(G\times G)p$ and $\pi'((G\times G)p)=Gx'$. The choice of $p\in \pi^{-1}(x)$, $x\in Gx$ and $x'\in Gx'$ are  irrelevant. Moreover,  if $\gamma:\bb R\to M$ is a geodesic orthogonal to orbits in $M$, the identification $\cal H\leftarrow \cal H''\to \cal H'$ sends $\gamma$ to a geodesic orthogonal to orbits in $M'$. We make it explicit in the following proposition (see Proposition 4.4 in \cite{speranca2016pulling} for the case of a fixed point $x$).
	
	\begin{proposition}\label{prop:fibradonormal}
		There is an isomorphism $\Phi:\nu Gx\to \nu Gx'$ satisfying:  if $\cal O\subset \nu Gx$ is such that $\exp|_\cal O$ is a diffeomorphism, then $\exp|_{\Phi(\cal O)}$ is a diffeomorphism. 
	\end{proposition}
	
	\subsection{Proof of Proposition \ref{prop:existence-invariant-metric}}
	We recall that a $G$-connection on a $G$-principal bundle is a differential $1$-form on $P$ with values on the Lie algebra $\mathfrak{g}$, which is $G$-equivariant and recognizes the Lie algebra generators of the fundamental vector fields on $P$. That is, for every $\xi\in\lie g$, $X\in TP$ and $q\in G$,
	\begin{itemize}
		\item $\omega(p\xi)=\xi$
		\item $\omega_{pg}(X g)=\Ad_{g^{-1}}\omega_p(X)$
	\end{itemize}
	Direct averaging  a connection 1-form by the $\star$-action gives an invariant form: let $\omega_0$ be a connection 1-form for $\pi:P\to M$. Given  a Haar measure $\mu$ on $G$ with unitary volume, define 
	\begin{equation}
	\omega_p(X) := \int_G(\omega_0)_{gp}(gX)d\mu.
	\end{equation}
	Since the $\bullet$- and $\star$-action commutes, it is immediate that $\omega$ is a connection 1-form for $\pi$ and that $\omega_{qp}(qX)=\omega_q(X)$ for all $X\in TP,~q\in G$.
	\halmos
	\subsection{Proof of Proposition \ref{prop:fibradonormal}}
	Given $x\in M$, let us describe an isomorphism between $\nu Gx$ and $\nu Gx'$. 
	
	\begin{claim}\label{claim:p}
		Given $x\in M$, there is $p_0\in \pi^{-1}(x)$ such that the isotropy subgroup $(G\times G)_{p_0}$ is the diagonal $\Delta G_x=\{(q,q)~|~q\in G_x\}$.
	\end{claim}
	\begin{proof}
		First of all, since $\pi(rps^{-1})=rp$, if $(r,s)\in (G\times G)_p$, then $r\in G_x$. On the other hand, we can write $P=\bigcup U_i\times G$, as the bundle defined  by the $\star$-cocycle $\{\phi_{ij}:U_i\cap U_j\to G \}$. Suppose $x\in U_i$. By definition, the $G\times G$-action on $U_i\times G$ is given by $r(x,g)s^{-1}=(rx,sgr^{-1})$. Therefore, for all $r\in G_x$, $r(x,1)r^{-1}=(rx,rr^{-1})=(x,1)$.  
	\end{proof}
	Given $p_0$ as in Claim \ref{claim:p},  let $\Psi:Gx\to P$ be defined as $\Psi(rx)=rp_0r^{-1}$.  Given $X\in TM$ and $p\in \pi^{-1}(x')$, denote by $\cal L_p(X)\in \cal H_p$ the unique horizontal vector such that $d\pi(\cal L_p(X))=X$. The maps $\Psi$ and $\cal L_p$ satisfies:
	\begin{equation}
	\Psi(ry)=r\Psi(y)r^{-1},\qquad \cal L_{rps^{-1}}(rX)=r\cal L_p(X)s^{-1}.
	\end{equation}
	Define the morphism $\Phi:\nu Gx\to \nu Gx'$ as
	\begin{equation}\label{eq:Phi}
	\Phi_{rx}(X)=d\pi'(\cal L_{\Psi(rx)}(X)).
	\end{equation}
	\begin{claim}\label{claim:2}
		$\Phi_{rx}$ is an isometry between $\cal H$ and $\cal H'$.  
	\end{claim}
	\begin{proof}
		It is sufficient to show that $\cal L_p(\cal H_x)=\cal H''_p$, since $\cal H''\subset \ker d\pi'$ and both $\cal L_p$ and $d\pi'_{(\ker d\pi')^\bot}$ are isometries. 
		%		Since $\cal H''\subset \ker d\pi$, it is sufficient to show that $d\pi(\cal H'')=\cal H$. 
		However, since $\pi(rps^{-1})=rx$, $\cal L_p$ sends vectors tangent to $Gx$ to vectors tangent to $(G\times G)p$. A dimension count shows that $\cal L_p(\cal H_x)=\cal H''_p$.
	\end{proof}
	Let $\cal O\subset \nu Gx$ be such that $\exp|_\cal O$ is a diffeomorphism.  Then $\tilde\Psi:\exp(\cal O)\times G\to P$, defined as
	\[\tilde\Psi(\exp_y(v),g)=\exp_{\Psi(y)}(\cal L_{\Psi(y)}(v))g^{-1},\] 
	is a trivialization along $\exp(\cal O)$. Moreover,
	\begin{align}
	\tilde\Psi(\exp_{ry}(rv),sgr^{-1})&=\exp_{\Psi(ry)}(\cal L_{\Psi(ry)}(rv))r(sg)^{-1}=\exp_{\Psi(ry)}(\cal L_{r\Psi(y)r^{-1}}(rv))r(sg)^{-1} \nonumber\\\label{eq:eqvPhi}&=\exp_{\Psi(ry)}(r\cal L_{\Psi(y)}(v)r^{-1})r(sg)^{-1}=r\tilde\Psi(\exp_{y}(v),g) s^{-1}.
	\end{align}
	Therefore, as in \eqref{eq:pi'}, $\pi'(\tilde\Psi(\exp_y(v),g))=g\exp_y(v)$, thus, arguments as in the proof of Theorem \ref{thm:star} guarantees that
	\[\exp_y(v)\mapsto \pi'(\tilde\Psi(\exp_y(v),\id)) \]
	defines a diffeomorphism. On the other hand, since $\pi'$ is a Riemannian submersion and $v\in(\ker d\pi')^\bot$, \begin{align*}
	\pushQED{\qed}
	\pi'(\tilde\Psi(\exp_{rx}(v),\id))&=\pi'\exp_{rxr^{-1}}(\cal L_{rxr^{-1}}(v))\\&=\exp_{\pi '(rxr^{-1})}(d\pi'\cal L_{rxr^{-1}}(v))=\exp_{rx'}(\Phi(v)).\qedhere
	\popQED		
	\end{align*}
	\begin{remark}
		\label{rmk:general}
		The orbit $Gx$ in Proposition \ref{prop:fibradonormal} can be replaced by any $G$-invariant submnifold $N$, provided there is a map $\Psi:N\to P$ satisfying $\Psi(rx)=r\Psi(x)r^{-1}$. The rest of the proof proceeds along the same lines.
		%\begin{equation}
		%\Psi(r\cdot x,qgr^{-1}) = r\star q\bullet \Psi(x,g).
		%\end{equation}
		%If $\nu N$ is the normal bundle of $N$ regarded with the $G$-action induced by $G$ and $V(N)$ is a invariant neighbourhood by such action of the zero section on $\nu N$, supposing that the exponential map is a diffeomorphism on its image when restricted to this neighbourhood, let $V = \pi^{-1}(\exp(V(N))).$ Then $\pi': V \to M'$ is the projection of a $G$-principal bundle whose image is equivariantely diffeomorphic to $\exp(V(N))$.
	\end{remark}
	\begin{remark}
		The equivalence of condition \eqref{eq:isotropy} and Definition \ref{defn:howtoconstruct} follows from the equivariance of $\tilde \Phi$ (equation \eqref{eq:eqvPhi}): let $P\to M$ be a (possibly non-special) $G$-$G$-bundle satisfying \eqref{eq:isotropy}. For every orbit $Gx\subset M$, there is an open tubular neighborhood $U_x$ of $G_x$ such that $\exp:\cal O\to U_x$ is a diffeomorphism for some $\cal O_x\subset \nu Gx$. Consider the trivialization $\tilde \Phi_x:U_x\times G\to P$ as in the proof of Proposition \ref{prop:fibradonormal}. The equivariance \eqref{eq:eqvPhi} guarantees that the transition functions related to the open cover $\{U_x\}_{x\in M}$ satisfies the equivariance condition \eqref{eq:conjugation}.
	\end{remark}

	%\begin{proposition}\label{existenceofconnection}
	%	If $P$ is equipped with a Kaluza-Klein metric induced by a $\star$-invariant connection on $P$ and a $G$-invariant metric on $M$, then $\pi ': P \to M$ is a submersion.
	%\end{proposition}
	
	%We can now prove the more general statement of the Proposition \ref{fibradonormal} (see Remark \ref{rmk:general}).
	%\begin{proof}
	%If $N$ is a submanifold of $M$ and $\nu N$ is a sub-bundle of the tangent bundle of $M$, then the metrics induced on $N$ and $\nu N$ are $\star$-invariants.
	%Let $\Psi : N\times G \to \pi^{-1}(N)$ as stated. Denote $(v,x) \in V(N)$ where $v\in \nu_xN$ and $x\in N$. Define $\tilde{\Psi} : V(N)\times G\to \pi^{-1}(V(N))$ by
	%\[\tilde \Psi((x,v),g) := \exp_{\Psi(x,g)}(\mathcal{L}_{\pi}v).\] Then
	%\begin{align*}
	%\tilde{\Psi}((rv,r\cdot x),gr^{-1}) &= \exp_{\Psi(r\cdot x,gr^{-1})}(\mathcal{L}_{\pi}rv)\\
	%&= \exp_{r\star \Psi(x,r)}(r\star \mathcal{L}_{\pi})\\
	%&= r\star \exp_{\Psi(x,g)}(\mathcal{L}_{\pi}v)\\
	%&= g\bullet r\star \exp_{\Psi(x,1)}(\mathcal{L}_{\pi}v) = r\star g \bullet \tilde{\Psi}((v,x),1).
	%\end{align*}
	%This proves that the map $\tilde \Psi$ is $G\times G$-equivariant. 
	%\end{proof}

	\section{Cheeger deformations and Ricci curvature}
	
	Lets recall briefly the process known as \emph{Cheeger deformation}. The main reference are the notes by W. Ziller \cite{ziller} on M. M\"uter's thesis. The exposition and Theorem \ref{thm:principal} does not intend to be original, but to make the article more self-contained.
	
	Let $(M,g)$ be a Riemannian manifold and $(G,Q)$ a compact Lie group with bi-invariant metric $Q$ acting by isometries on $M$. For each $p \in M$ we can decompose orthogonally the Lie algebra $\mathfrak{g}$ of $G$ as $\mathfrak{g} = \mathfrak{g}_p\oplus\mathfrak{m}_p$, where $\mathfrak{g}_p$ denotes the lie algebra of the isotropy group at $p$. Observe that $\mathfrak{m}_p$ is isomorphic to the tangent space to the orbit $G\cdot p$. 
	
	We call the tangent space $T_pG\cdot p$ as the \emph{vertical space} at $p$, $\mathcal{V}_p$. Its orthogonal complement, $\cal H_p$   is the \emph{horizontal space}. Given an element $U\in \lie g$, we define its action vector
	\begin{equation*}
	U^*_p=\frac{d}{dt}\Big|_{t=0}e^{tU}\cdot p.
	\end{equation*}
	The map $U\mapsto U^*_p$ defines linear morphism whose kernel is $\lie g_p$ (the map $U\mapsto U^*$ defines a Lie algebra morphism between $\lie g$ and the algebra of vector fields $\lie X(M)$ with the vector field brackets). In particular, a tangent vector $\overline X$ at $p$ can be uniquely decomposed  as $\overline X = X + U^{\ast}$, where $X$ is horizontal and $U\in \lie m_p$.
	
	The main idea in the Cheeger deformation is to consider the product manifold $M\times G$, observing that the action 
	\begin{equation}\label{eq:actioncheeger}
	r(p,g) := (r\cdot p, gr^{-1})
	\end{equation}
	(compare with \eqref{eq:staraction}) is isometric for the product metric $g\times \frac{1}{t}Q$, $t > 0$. Action \eqref{eq:actioncheeger} is free and its quotient space is diffeomorphic to $M$ (see \eqref{eq:pi'} and \eqref{eq:iota}). In fact, the projection
\begin{align}\label{eq:pi'cheeger}
	\pi' : M\times G &\to \star \big \backslash M\times G\\
	(p,g) &\mapsto  g\cdot p \nonumber
\end{align}
	identifies $\star \big\backslash M\times G$ with $M$, inducing a family of metrics $g_t$ on $M$. We proceed defining some important tensors:
	\begin{itemize}
		\item let $P : \mathfrak{m}_p \to \mathfrak{m}_p$ be the \emph{orbit tensor} of the $G$-action, defined by
		\[g(U^{\ast},V^{\ast}) = Q(PU,V),\quad\forall U^{\ast}, V^{\ast} \in \mathcal{V}_p.\] 
		$P$ is a symmetric and positive definite operator 
		%		that can be extended to $\mathfrak{g}$ by declaring $P(\mathfrak{g}_p) = 0$
		\item denote by $P_t:\lie m_p\to \lie m_p$, the operator 
		\[g_t(U^{\ast},V^{\ast}) = Q(P_tU,V), \quad\forall U^{\ast}, V^{\ast} \in \mathcal{V}_p\]
		\item  we define the \emph{metric tensor of $g_t$},  $C_t:T_pM\to T_pM$ as
		\[g_t(\overline{X},\overline{Y}) = g(C_t\overline{X},\overline{Y}), \quad\forall \overline{X}, \overline{Y} \in T_pM\]
	\end{itemize}
	\begin{proposition}[Proposition 1.1 in \cite{ziller}]\label{propauxiliar}The tensors above satisfy:
		\begin{enumerate}
			\item $C_t = 1$ in $\mathcal{H}_p$, $C_t = P^{-1}P_t \in \mathcal{V}_p$
			\item $P_t = (P^{-1} + t1)^{-1} = P(1 + tP)^{-1}$
			\item If $\overline{X} = X + U^{\ast}$ then $C_t(\overline{X}) = X + (1 + tP)^{-1}U^{\ast}$
		\end{enumerate}
	\end{proposition}
	
	The advantage of Cheeger deformations is that $g_t$ does not produce `new' planes with zero curvature -- in fact, up to a reparametrization of $Gr_2(TM)$ (the Grassmannian of 2-planes of $TM$), the sectional curvature of a fixed plane is non-decreasing. Let $\overline{X} = X + U^{\ast}, \overline{Y} = Y + V^{\ast}$  tangent vectors. Then $\kappa_t(X,Y) := K_{g_t}(C_t^{-1}\overline{X},C_t^{-1}\overline{Y})$ satisfies
	\begin{equation}\label{eq:curvaturaseccional}
	\kappa_t(\overline X,\overline{Y}) = \kappa_0(\overline{X},\overline{Y}) +\frac{t^3}{4}\|[PU,PV]\|_Q^2 + z_t(\overline{X},\overline{Y}),
	\end{equation}
	where $z_t$ is a non-negative term related to the fundamental tensors of the foliation (see \cite[Proposition 1.3]{ziller}).
	%At the principal part of the action, i.e., at the points where the action is free,
	%\begin{equation}\label{ilustrativo}
	%z(\overline{X},\overline{Y},t) = 3\Big\|\dfrac{t^{1/2}}{(1 + tP)^{1/2}}\Big (P(A(X,Y) - S_XV^{\ast} - S_YU^{\ast} + \nabla_{U^{\ast}}{V^{\ast}}^{\mathcal{V}}) - \frac{t}{2}[PU,PV]\Big) \Big\|^2_Q,
	%\end{equation}
	%with $A$ and $S$ the O'Neill tensor and the shape operator. We note in particular that $\kappa(t)$ is non-decreasing.
	In what follows, we  study the behavior of Cheeger deformations on the Ricci curvature. 
	
	Fix a point $p \in M$ and consider $\{e_1,\ldots,e_k,e_{k+1},\ldots,e_n\}$ a $g$-orthonormal base for $T_pM$ such that $e_{k+1},\ldots,e_n$ are horizontal and, for $i\leq k$, $e_i = \lambda_i^{-1/2}v^{\ast}_i$ where $v_1,\ldots,v_k$ is a set of $Q$-orthonormal eigenvectors of $P$ with eigenvalues $\lambda_1,\ldots,\lambda_k$,  taken in non-decreasing order. Observe that $\{C_t^{-1/2}e_i\}_{i=1}^n$ is a $g_t$-orthonormal base for $T_pM$. From Proposition \ref{propauxiliar}, $C_t^{-1/2}e_i = (1+t\lambda_i)^{1/2}e_i$ for $i \le k$ and $C_t^{-1/2}e_i = e_i$ for $i > k.$
	
	Define $ \Ricci^{\mathcal{H}}(\overline{X}) := \sum_{i=k+1}^n\kappa_0(\overline{X},e_i)$. We show that $\Ricci_g^{\cal H}$ and the curvature of each orbit (seem as a normal homogeneous space) determine the Ricci curvature. From \eqref{eq:curvaturaseccional} we have
	
	\begin{multline}
	\Ricci_{g_t}(C_t^{-1}\overline{X}) = \sum_{i=1}^{n}R_{g_t}(C_t^{-1/2}e_i,C_t^{-1}\overline{X},C_t^{-1}\overline{X},C_t^{-1/2}e_i) =\sum_{i=1}^n\kappa_t(C_t^{1/2}e_i,\overline{X})\\
	= \sum_{i=1}^n\kappa_0(C_t^{1/2}e_i,\overline{X}) + \sum_{i=1}^nz_t(C_t^{1/2}e_i,\overline{X}) + \frac{t^3}{4}\sum_{i=1}^k\|[PC_t^{1/2}\lambda_i^{-1/2}v_i,P\overline{X}]\|_Q^2
	\\= \Ricci_g^{\mathcal{H}}(\overline{X}) + \sum_{i=1}^nz_t(C_t^{1/2}e_i,\overline{X}) + \sum_{i=1}^k\frac{1}{1+t\lambda_i}\Big(\kappa_0(\lambda_i^{-1/2}v^{\ast}_i,\overline{X}) + \frac{\lambda_it^3}{4}\|[v_i,PU]\|_Q^2 \Big).
	\end{multline}
	
	In particular,
	\begin{align*}
	\Ricci_{g_t}(\overline{X}) =& \Ricci_g^{\mathcal{H}}(C_t\overline{X}) + \sum_{i=1}^nz_t(C_t^{1/2}e_i,C_t\overline{X}) \\&+ \sum_{i=1}^k\frac{1}{1+t\lambda_i}\Big(\kappa_0(\lambda_i^{-1/2}v^{\ast}_i,C_t\overline{X}) + \frac{\lambda_it}{4}\|[v_i,tP(1+tP)^{-1}U]\|_Q^2\Big).
	\end{align*}
	On the other hand, using Proposition \ref{propauxiliar} we get:
	\begin{align}
	\lim_{t\to \infty}\Ricci_g^{\mathcal{H}}(C_t\overline{X})&=\Ricci_g^{\mathcal{H}}({X}) \label{eq:horizontalRicterm}\\
	\lim_{t\to\infty}\sum_{i=1}^k\frac{\lambda_i}{1+t\lambda_i}\kappa_0(v^{\ast}_i,C_t\overline{X})&=0\\
	\lim_{t\to\infty}\sum_{i=1}^k\frac{t\lambda_i}{4+4t\lambda_i} \|[v_i,tP(1+tP)^{-1}U]\|_Q^2&=\sum_{i=1}^k\frac{1}{4}\|[v_i,U]\|_Q^2\label{eq:orbit}
	\end{align}
	%\begin{equation}\label{eq:limite}\lim_{t\to \infty}\Ricci_{g_t}(\overline{X}) = \lim_{t\to \infty}\Ricci_{g_t}(X + U^{\ast}) = \Ricci^{\mathcal{H}}(X) + \frac{1}{4}\sum_{i=1}^k\|[v_i,U]\|_Q^2.\end{equation}
	
	Consider $G/G_p$ as a normal homogeneous space where $G$ is endowed with the metric defined by $Q$. Then \eqref{eq:orbit} is exactly the Ricci curvature of $G/G_p$. In particular, there is a constant $K>0$ such that $\sum_{i=1}^k\frac{1}{4}\|[v_i,U]\|_Q^2\geq K||U||^2_Q$, provided $G/G_p$ has finite fundamental group.
	
	With \eqref{eq:horizontalRicterm}-\eqref{eq:orbit} at hand, a compactness argument shows that $g_t$ has positive Ricci curvature, provided its orbits have finite fundamental group and $\Ricci_g^\cal H(X)\geq 0$ for all non-zero $X\in\cal H$.
	
	\begin{theorem}\label{thm:principal}
		Let $(M,g)$ be a compact Riemannian $G$-manifold, where $G$ is a compact. Suppose the orbits of $M$ have finite fundamental group and $\Ricci_g^{\mathcal{H}}(X) > 0$ for all non-zero  $X \in \mathcal{H}$. Then $g_t$ has positive Ricci curvature for some $t>0$. 
	\end{theorem}
	\begin{proof}
		Define $F(t,\overline{X})=\Ricci_{g_t}(\overline{X})$. We argue by contradiction, supposing that there is no $t>0$ such that $g_t$ has positive Ricci curvature. Therefore, for each time $t=n\in \bb N$, there is a point $p_n$ and a unitary vector $\overline{X}_n=X_n+U_n^*$ such that $F(p_n,t_n,\overline{X}_n)\leq 0$. Passing to a subsequence, if necessary, we conclude that the limit $\lim\overline{X}_n=\overline{X}$ satisfy
		\[0\geq \lim_{n\to \infty}F(t_n,\overline{X}_n)=\Ricci_g^{\mathcal{H}}({X})+\sum_{i=1}^k\frac{1}{4}\|[v_i,U]\|_Q^2>0,\]
		a contradiction.\end{proof}
	
	\subsection{The geometry of cross-diagrams}
	We now apply Theorem \ref{thm:principal} to $G$-$G$-bundles. Let $M \leftarrow P \to M'$ be a $G$-$G$-bundle. Proposition \ref{prop:existence-invariant-metric} guarantees that, given a $G$-invariant metric $g_M$ on $M$, there is a $G{\times}G$-invariant metric $g_P$ on $P$ and (given $g_P$) a unique metric $g_{M'}$ on $M'$ that makes $\pi'$ a submersion.

	%	
	
	%	we shall now apply Theorem \ref{thm:principal} and Cheeger deformations to ensure metrics of positive Ricci curvature on $M'$. We recall that if $p : (N,g_N) \to (B,g_B)$ is a submersion, then there are two fundamental tensors associated to it. The first one, denoted by $A$, is the \emph{O'Neill tensor}, defined by
	%	\begin{equation}\label{tensordohomi}
	%	A(X,Y) := \frac{1}{2}[X,Y]^{\mathcal{V}},
	%	\end{equation}
	%	where $X, Y$ are $p$-horizontals and $[X,Y]^{\mathcal{V}}$ is $p$-vertical.
	%	The second one, denoted by $S$ and entitled \emph{shape operator}, is defined by
	%	\begin{equation}S_XU = -(\nabla_UX)^{\mathcal{V}},\end{equation}
	%	where $X$ is a $p$-horizontal vector and $U$ a $p$-vertical vector.
	Denote by $\cal L_\pi$ and $\cal L_{\pi'}$ the horizontal lifts of $\pi$ and $\pi'$, respectively, and  $\cal V^\pi=\ker d\pi$, $\cal V^{\pi'}=\ker d\pi'$. We recall the classical O'Neill formula for both $\pi$ and $\pi'$ (see O'Neill \cite{oneill} or Gromoll--Walschap \cite{gw} for details). Let $X,Y\in(\cal V^{\pi})^\bot$, $X',Y'\in(\cal V^{\pi'})^\bot$. We have
		\begin{align}\label{oneillformula}
		K_{g_M}(X,Y) &= K_{g_P}(\mathcal{L}_{\pi}X,\mathcal{L}_{\pi}Y) + \frac{3}{4}\|[{\mathcal{L}_{\pi}X},\mathcal{L}_{\pi}Y]^{\mathcal{V}^{\pi}}\|^2_{g_P},
		\\K_{g_{M'}}(X',Y') &=  K_{g_P}(\mathcal{L}_{\pi'}X,\mathcal{L}_{\pi'}Y) + \frac{3}{4}\|[{\mathcal{L}_{\pi'}X},\mathcal{L}_{\pi'}Y]^{\mathcal{V}^{\pi'}}\|^2_{g_P}.
		\end{align}
	Recall from the proof of Claim \ref{claim:2} that, given $p\in P$, $\cal L_p|_{\cal H_{\pi(p)}}:\cal H_{\pi(p)}\to \cal H''_p$ and $d\pi'|_{\cal H''_p}:\cal H''_p\to \cal H'_{\pi'(p)}$ are isometries  (we keep the notation in section \ref{sec:geometry}, denoting $\cal H$, $\cal H'$ and $\cal H''$ the orthogonal to the $G$-orbits in $M$, $M'$ and the $G{\times}G$-orbit on $P$). Being $\cal H''$ orthogonal to both $\cal V^\pi$ and $\cal V^{\pi'}$,  we get
		\begin{equation}
		K_{g_{M'}}(d\pi'\mathcal{L}_{\pi}X,d\pi'\mathcal{L}_{\pi}Y) =  K_{g_M}(X,Y) + \frac{3}{4}\Big\{\|[{\mathcal{L}_{\pi}X},\mathcal{L}_{\pi}Y]^{\mathcal{V}^{\pi'}}\|^2_{g_P} - \|[{\mathcal{L}_{\pi}X},\mathcal{L}_{\pi}Y]^{\mathcal{V}^{\pi}}\|^2_{g_P}\Big\}.
		\end{equation}
	
	We claim we can change the metric on $P$ without affecting the curvatures $K_{g_{M'}}(d\pi'\mathcal{L}_{\pi}X,d\pi'\mathcal{L}_{\pi}Y),~  K_{g_M}(X,Y)$, but with  $\|[{\mathcal{L}_{\pi}X},\mathcal{L}_{\pi}Y]^{\mathcal{V}^{\pi}}\|^2_{g_P}$ arbitrarily small. Precisely:
		\begin{proposition}\label{prop:util}
			Let $M\leftarrow P\to M'$ be a $G$-$G$-bundle with $P$ compact. Let $g_P$ be a $G{\times}G$-invariant metric on $P$. Then, given $\epsilon > 0$, there exists $t > 0$ such that the metric $g_{P_t}$, obtained by a finite Cheeger deformation on the $G{\times }G$-manifold $P$, satisfy: for each pair $X,Y\in \cal H''$, 
			\[	K_{g_{M'_t}}(d\pi'X,d\pi'Y) -  K_{g_{M_t}}(d\pi X,d\pi Y)\geq -\epsilon||X\wedge Y||_{g_P}^2,\]
			where $g_{M_t}$, $g_{M'_t}$ are the resulting submersion metrics on $M$ and $M'$, respectively.
		\end{proposition}
\begin{proof}		Observe that, if $g_P$ is $G{\times}G$-invariant, so it is $g_{P_t}$. In fact, the product metric $g_P\times t^{-1}Q$ on $P\times G$ is $G{\times}G{\times}G$-invariant, where $G{\times}G{\times}G$ acts as
		\[(r,s,q)(p,g)=(rpq^{-1},sgq^{-1}).\]
		Action \eqref{eq:actioncheeger} is given by the $q$-coordinate, $s$ and $r$ descends to the $\pi$-principal and $\star$-actions on the quotient of $P\times G$ by $q$ (which is identified with $P$ via \eqref{eq:pi'cheeger}). Therefore, the resulting metric $g_{P_t}$ is invariant both under the $\pi$-principal action and the $\star$-action. 
		
		\begin{remark}
			It is worth remarking that the resulting metric on $M'$, $g_{M'_t}$, is a Cheeger deformation as well.
		\end{remark}
		
		To prove Proposition \ref{prop:util}, note that the vector $[X,Y]^{\mathcal{V}^{\pi}}$ does not change with Cheeger deformation, since  Cheeger deformation does not affect the horizontal (or vertical) space. Therefore, Proposition \ref{prop:util} follows from Lemma \ref{util} below. We set some notation first.
		
		%	To achieve our goal of producing metrics of positive Ricci curvature on $M'$, we compare the horizontal Ricci curvatures of $M$ and $M'$ through \eqref{oneillformula}. We produce positive horizontal Ricci curvature on $M'$ after a finite Cheeger deformation on $g_P$. The main idea is to use the fact that Cheeger deformations do not decreases curvature, so we can perform such procedure on $P$, regarded with product action $G\times G$, that is well defined since the actions $\bullet$ and $\star$ commute, in order to compress the length of vertical vectors on each point, according to the product action.
		
		Given $p\in P$, let  $\mathcal{P}_p$ be the orbit tensor associated to the $\pi$-principal action. Given a $G{\times} G$-invariant metric $g_P$ on $P$, there exists an orthonormal basis of $\mathcal{P}_p$-eigenvector    $\{v_1,...,v_k\}$ for $\mathcal{V}_p^{\pi}$. We denote by $\lambda_i$ the eigenvalue associated  to $v_i$. Since the principal action is free and  $P$ is compact, there is positive number $\lambda := \min_{p\in P, i}\lambda_i$.	
		
		\begin{lemma}\label{util}
		Given  $\epsilon > 0$, there exists $t > 0$ such that $g_{P_t}$ satisfies, for every $V\in\cal V^{\pi}$,
		\[g_{P_t}(V,V) \leq \epsilon\|V\|^2_{g_P}.\]
	\end{lemma}
	\begin{proof}
	 Given $p \in P$, each unitary  $V\in\cal V^\pi_p$  can be written uniquely as $V = \sum_i g_P(X,v_i)v_i$. In particular,
		\[g_{P_t}(V,V) = g_{P}((1 + t\mathcal{P})^{-1}V,V) = \sum_{i,j}g_P(V,v_i)g_P(V,v_j)g_P((1+ t\lambda_i)^{-1}v_i,v_j).\]
		Therefore,
		\[g_{P_t}(V,V) = \sum_i\dfrac{g_P(V,v_i)^2}{1 + t\lambda_i}.\]
         Once for each $i$ we have that ${\lambda_i} \ge \lambda$ we obtain
		\[g_{P_t}(V,V) = \sum_i\frac{g_P(V,v_i)^2}{1 + t\lambda_i} \le \sum_ig_P(V,v_i)^2\dfrac{1}{1+t\lambda} = \dfrac{1}{1+t\lambda}.\] Since we would like that
		\[ \dfrac{1}{1+t\lambda} \le \epsilon,\]
		we must have
		\[t \ge \dfrac{(1-\epsilon)}{\lambda \epsilon}.\qedhere\]
	\end{proof}
%	
%	Let $x \in M$ and $X, Y \in \mathcal{H}_x$. 
%	%	If $\mathcal{L}_{\pi}X$ and $\mathcal{L}_{\pi}Y$ denotes the $\pi$-horizontal lifts of $X$ and $Y$, then both are $\pi'$-horizontal. For such vectors w
%	We can compare the Sectional Curvatures of $M$ and $M'$ by:
%	\begin{align}
%	K_{g_M}(X,Y) &= K_{g_P}(\mathcal{L}_{\pi}X,\mathcal{L}_{\pi}X) + \frac{3}{4}\|[{\mathcal{L}_{\pi}X},\mathcal{L}_{\pi}X]^{\mathcal{V}^{\pi}}\|^2_{g_P},
%	\\K_{g_M'}(d\pi'\mathcal{L}_{\pi}X,d\pi'\mathcal{L}_{\pi}X) &=  K_{g_P}(\mathcal{L}_{\pi}X,\mathcal{L}_{\pi}X) + \frac{3}{4}\|[{\mathcal{L}_{\pi}X},\mathcal{L}_{\pi'}X]^{\mathcal{V}^{\pi'}}\|^2_{g_P}.
%	\end{align}
%	We conclude that
%	\begin{equation}
%	K_{g_M'}(d\pi'\mathcal{L}_{\pi}X,d\pi'\mathcal{L}_{\pi}X) =  K_{g_M}(X,Y) + \frac{3}{4}\Big\{\|[{\mathcal{L}_{\pi}X},\mathcal{L}_{\pi}X]^{\mathcal{V}^{\pi'}}\|^2_{g_P} - \|[{\mathcal{L}_{\pi}X},\mathcal{L}_{\pi}X]^{\mathcal{V}^{\pi}}\|^2_{g_P}\Big\}.
%	\end{equation}
    	Finally, we recall that the map 
	\[(X,Y)\mapsto \|[X,Y]^{\mathcal{V}^{\pi}}\|^2_{g_P}\]
	is tensorial on $X,Y\in\cal (V^{\pi})^\bot$ (see, for instance, the definition of the $A$-tensor on \cite{oneill} or \cite{gw}). Therefore, assuming $P$ compact, there is a constant $C>0$ such that   \[\|[X,Y]^{\mathcal{V}^{\pi}}\|^2_{g_P}\leq C||X||_{g_P}^2||Y||^2_{g_P}.\]
    Observing that $g_{P_t}$ induces the same metric $g_M$ for every $t$, we conclude that, for any given $\epsilon>0$ and any $G$-invariant metric $g_M$, there is a metric $g_{M'}$ such that
    \begin{multline*}
        	K_{g_{M'}}(d\pi'X,d\pi'Y) =  K_{g_M}(d\pi X,d\pi Y) + \frac{3}{4}\Big\{\|[X,Y]^{\mathcal{V}^{\pi'}}\|^2_{g_P} - \|[{X},Y]^{\mathcal{V}^{\pi}}\|^2_{g_P}\Big\}
		\\\geq  K_{g_{M}}(d\pi X,d\pi Y) -{\epsilon}||X\wedge Y||_{g_P}^2.\qedhere
		   \end{multline*}
\end{proof}
		   
	Taking $\{e_1, \ldots, e_k\}$ an orthonormal base for $\mathcal{H}_p''$ and $X\in\cal H_{\pi'(p)}$,  we have:
	\begin{equation}\label{curvaturadericcihorizontalexemplo} \Ricci^{\mathcal{H}}_{g_{M'}}(X)  \geq \Ricci^{\mathcal{H}}_{g_{M}}(d\pi\mathcal{L}_{\pi'}X)-\epsilon||X||_{g_{M'}}^2.\end{equation}
	
	Recalling that  $M'\leftarrow P\to M$ is a $G$-$G$-bundle as well, we conclude:
	
	\begin{theorem}\label{thm:moritaRicci}
	Let $M'\leftarrow P\to M$ be a $G$-$G$-bundle ans suppose $P$ compact. Then, $M$ has a $G$-invariant metric of positive horizontal Ricci curvature if and only if $M'$ does so.
	\end{theorem}

	Positive Ricci curvature on the \textit{exotic} projective planes $\Sigma P\bb C^{2m-1}$, $\Sigma P\bb H^{2m-1}$ follows from Theorem \ref{thm:principal} since vectors orthogonal to the orbits of the $U(m)$- and $Sp(m)$-action lift to vectors orthogonal to the $O(n)$-orbits on $\Sigma^{2n-1}$. In fact, positive horizontal Ricci curvature on $\Sigma^{2n-1}$ (provided by Theorem \ref{thm:moritaRicci}) ensures positive horizontal Ricci curvature on  $\Sigma P\bb C^{2m-1}$ and $\Sigma P\bb H^{2m-1}$. The conclusion follows from the fact the $U(m)$-, respectively, $Sp(m)$-orbits have finite fundamental group.

	\bibliographystyle{plain}
	
	\bibliography{main}

\end{document}